\documentclass[11pt,twoside,letter]{article}
\usepackage{amsthm, amsmath, amssymb, amsfonts, graphicx, epsfig}
\usepackage{mathrsfs}
\usepackage{bbm}
\usepackage{enumerate}
\usepackage[margin=1.0in, letterpaper]{geometry}
\usepackage{graphicx}
\usepackage[font=sl,labelfont=bf]{caption}
\usepackage{subcaption}
\usepackage{epstopdf}
\usepackage{cancel}

\usepackage[colorlinks=true, pdfstartview=FitV, linkcolor=blue,
citecolor=blue, urlcolor=blue]{hyperref}
\usepackage[usenames]{color}

\usepackage{url}
\makeatletter\def\url@leostyle{%
	\@ifundefined{selectfont}{\def\UrlFont{\sf}}{\def\UrlFont{\scriptsize\ttfamily}}} \makeatother

\newtheorem{theorem}{Theorem}
\newtheorem{proposition}[theorem]{Proposition}
\newtheorem{lemma}[theorem]{Lemma}

\theoremstyle{definition}

\theoremstyle{remark}
\newtheorem{remark}[theorem]{Remark}

\numberwithin{equation}{section}
\numberwithin{theorem}{section}

\def\cN{\mathcal{N}}

\def\cQ{\mathcal{Q}}

\def\bE{\mathbb{E}}
\def\bN{\mathbb{N}}
\def\bP{\mathbb{P}}
\def\bR{\mathbb{R}}

\def\sF{\mathscr{F}}

\newcommand{\wt}{\widetilde}
\newcommand{\wh}{\widehat}

\newcommand{\set}[1]{\{#1\}}            

\DeclareMathOperator*{\w}{w-\!\!} 
\DeclareMathOperator{\dif}{d \!}        

\title{Parameter estimation for discretely sampled stochastic heat equation driven by space-only noise}

\author{
	Igor Cialenco\,\thanks{Department of Applied Mathematics, Illinois Institute of Technology
		\newline \hspace*{1.45em}  10 W 32nd Str, Building RE, Room 220, Chicago, IL 60616, USA
		\newline \hspace*{1.45em}  Emails: \url{cialenco@iit.edu}, URL: \url{http://cialenco.com}
		\vspace{0.5em}} ,
	\and
	Hyun-Jung Kim\,\thanks{Department of Mathematics, University of California Santa Barbara, 
		\newline \hspace*{1.45em} South Hall, Room 6607,
		University of California Santa Barbara, CA 93106, USA
		\newline \hspace*{1.45em} Email: \url{hjkim@ucsb.edu},
		URL: \url{https://sites.google.com/view/hyun-jungkim}
		\vspace{0.5em}}
}

\date{
{\small 
		First Circulated March 18, 2020 \\
		This Version: July 12, 2021 }}
	
\begin{document}
	
	\maketitle	
	
	\vspace{-2em}
	
	{\footnotesize
		\begin{tabular}{l@{} p{350pt}}
			\hline \\[-.2em]
			\textsc{Abstract}: \ & 

We derive consistent and asymptotically normal estimators for the drift and volatility parameters of the stochastic heat equation driven by an additive space-only white noise when the solution is sampled discretely in the physical domain. We consider both the full space and the bounded domain. We establish the exact spatial regularity of the solution, which in turn, using power-variation arguments, allows building the desired estimators. We show that naive approximations of the derivatives appearing in the power-variation based estimators may create nontrivial biases, which we compute explicitly. The proofs are rooted in Malliavin-Stein's method. 

\\[0.5em]
\textsc{Keywords:} \ &  SPDEs, parabolic Anderson model, quadratic variation, parameter estimation, 
discrete sampling, space-only noise, Malliavin calculus, Stein's method, central limit theorem. \\
\textsc{MSC2010:} \ &  Primary 60H15; Secondary 60H07, 62M99\\[1em]
\hline
\end{tabular}
	}

\section{Introduction}\label{sec:intro}
It is well recognized by now that stochastic partial differential equations (SPDEs) serve as a key modeling tool in various applied disciplines, and we refer the reader to 
the classical monographs \cite{RozovskiiBook,DaPratoZabczykBook2014,LototskyRozovsky2018Book}, and also to the textbooks \cite{Chow2007Book,Hairer2009Book,LototskyRozovsky2017Book}, for an in depth discussion of the theory of SPDEs and their numerous applications.  As with any stochastic model, statistical analysis of SPDEs driven models is of fundamental practical importance, and it has been investigated in numerous works; cf. the survey paper \cite{Cialenco2018}. The existing literature on statistical inference for SPDEs usually deals with space-time noise. In \cite{CialencoKimLototsky2019}, the authors make the first attempt to study inverse problems for stochastic evolution equations driven by \textit{space-only noise}, and the present paper contributes, in particular, to the efforts initiated therein. Specifically, we consider the one-dimensional stochastic heat equation driven by an additive space-only noise 
\begin{equation}\label{eq:1dPAM}
\begin{cases}
\frac{\partial u(t,x)}{\partial t}=\theta \frac{\partial^2 u(t,x)}{\partial x^2}+\sigma\dot W(x),
\quad t>0,\ x\in G \subset \mathbb{R},\\
u(0,x)=0,\quad x\in G,
\end{cases}
\end{equation}
where $\theta>0$, $\sigma\in \mathbb{R}$, and $\dot W(x)$ is a space-only Gaussian white noise on a complete probability space $(\Omega,\sF,\bP)$. That is, $\dot W$ is a Gaussian process on $G$ with  mean zero and covariance $\mathbb{E}\left[\dot W(x)\dot W(y)\right]=\delta(x-y)$, for every $x,y\in G$, and where $\delta$ is the Dirac-delta function.  The main goal of this paper is to estimate the drift coefficient $\theta$ or the diffusion (volatility) parameter $\sigma$, based on \textit{discrete observations} of the solution $u$.  

In \cite[Section~3]{CialencoKimLototsky2019}, statistical analysis of \eqref{eq:1dPAM} was performed within the so-called spectral approach, when the observations are  performed in the Fourier space and discretely in time. It was proved that the parameter $\theta$ can be determined exactly by knowing the value of just one Fourier coefficient of the solution $u$ at three time points. Finding the drift coefficient exactly without any statistical analysis is an unusual feature, and it is due to the singular nature of the statistical experiment at hand, which in turn in this case was due to the special space-only structure of the noise. Similar effect was noticed in \cite{CialencoLototsky2009} for equations driven by time-only noise. On the other hand, within the spectral approach to \eqref{eq:1dPAM}, the parameter $\sigma$ can only be estimated consistently by observing a larger number of Fourier coefficients of $u$, although only at two time points. While these are indeed remarkable properties, we emphasize that usually, the observer would measure the solution $u$ itself, rather than its Fourier coefficients. Moreover, to find or approximate even one Fourier coefficient of $u$ one has to use the values of $u$ on the whole space $G$ rather than some local space-time values; see \cite{CialencoDelgado-VencesKim2019} for more details on approximating the maximum likelihood estimators obtained by spectral approach. 

The parameter estimation problem for  SPDEs when the solution is discretely sampled in space and/or time component was  addressed systematically only recently by quite different methods; cf.  \cite{CialencoHuang2017,BibingerTrabs2017,BibingerTrabs2019,Chong2019,CialencoDelgado-VencesKim2019,Chong2019a,KainoMasayukiUchida2019,KhalilTudor2019} and references therein. We follow the $p$-variation approach of \cite{CialencoHuang2017}, which  exploits the optimal regularity of the solution combined with the power variation calculus to build consistent and asymptotically normal estimators of the unknown parameters; see also \cite{PospivsilTribe2007,ZiliZougar2019,KhalilTudor2019,CKP21}. 

The \textbf{main contributions} of this paper can be summarized as follows. \textbf{First}, in Section~\ref{sec:regularity}, we study the regularity properties of the solution $u$ to \eqref{eq:1dPAM} on the whole space $G=\bR$, and prove that $\partial_x u(t,x)=:u_x(t,x)$ is H\"older continuous, with exponent almost $1/4$ in $t$, and almost $1/2$ in $x$. The case of a bounded domain $G$ was investigated in \cite{KimLototsky2017}. \textbf{Second}, assuming that the derivative $u_x(t,x)$ is sampled at a uniform space grid $a=x_0<x_1<\cdots<x_N=b$, and $M$ time points, using the quadratic variation of $u_x$ in $x\in[a,b]$, we derive estimators for $\theta$ and $\sigma$. In Section~\ref{sec:discrtDerivU} we prove consistency  and asymptotic normality as the number of spatial points increases, $N\to\infty$.  The case $M=1$ and $G=(0,\pi)$ was discussed in  \cite{CialencoKimLototsky2019}. \textbf{Third}, and more importantly, in Section~\ref{sec:discrtSolU}, we show that naively substituting $u_x$ by its finite-difference approximations in the aforementioned estimators introduces a (nontrivial) bias. We compute explicitly this bias, and  prove that the corresponding estimators based on the discrete sampling of the solution $u$ in physical domain are also consistent and asymptotically normal as the spatial sample size increases. 
We argue that this methodology to construct consistent estimators can be used to study parameters estimation problems for SPDEs driven by colored noises, and we refer to the recent work \cite{CKP21} that treats semi-linear SPDEs. 
\textbf{Fourth}, we treat both, bounded and unbounded domains $G\in\bR$. Interestingly, the constructed estimators are the same for bounded and unbounded domains. We note that the spectral approach in principle cannot be applied to the unbounded domain case. 

The methods of proof of statistical properties of the proposed estimators are rooted in Malliavin calculus and Stein's method, combined with tight control of the covariance structures of the relevant spatial increments of the solution. The theoretical results are illustrated via numerical experiments in Section~\ref{sec:numerical}.  Some concluding remarks are made in Section~\ref{sec:future}. For reader's convenience, technical and lengthy proofs are deferred to appendix.

\section{Analytical properties of the solution}\label{sec:regularity}
In this section we will derive some fine regularity properties of the solution of \eqref{eq:1dPAM}. These properties, besides being of independent interest, will be also conveniently used in Section~\ref{sec:stats} to derive estimators for the unknown parameters $\theta$ and $\sigma$. 

We will consider two cases for the domain $G$: (a) the whole space, namely $G=\bR$; (b) the bounded domain $G=(0,\pi)$ for which we endow equation \eqref{eq:1dPAM} with Dirichlet boundary conditions $u(t,0)=u(t,\pi)=0, \ t >0$. Formally, when $G=\bR$, one can write $\dot W(x)=\displaystyle\frac{\partial W(x)}{\partial x}$, where $W$ is a two-sided Brownian motion.

We define a mild solution of \eqref{eq:1dPAM} as
\begin{equation}\label{eq:mild-sol}
u(t,x)
=\sigma\int_0^t\int_{G}\mathcal{P}(t-s,x,y)\dif W(y)\dif s,
\end{equation}
where $\mathcal{P}$ is the fundamental solution of the corresponding deterministic heat equation. That is
\begin{itemize}
\item[(a)] $\mathcal{P}(t,x,y):=P(t,x-y)=(4\pi \theta t)^{-1/2}e^{-|x-y|^2/4\theta t}$, if $G=\mathbb{R}$;
\item[(b)] $\mathcal{P}(t,x,y):=P^D(t,x,y)=\displaystyle\frac{2}{\pi}\sum_{k=1}^{\infty} e^{-k^2\theta t}\sin(kx)\sin(ky)$, if $G=(0,\pi)$ with Dirichlet boundary conditions.
\end{itemize}
One can show that $u$ is a well-defined  Gaussian field; we refer to \cite[Chapter 3]{Khoshnevisan2014BookSBMS} for $G=\mathbb{R}$,  and  \cite[Chapter 5]{KimLototsky2017} when $G=(0,\pi)$. As mentioned in  \cite{KimLototsky2017} all analytical properties of the solution $u$ remain true when $G=(0,\pi)$ with Neumann boundary conditions.

For $0<\alpha\leq 1$, we will denote by $\mathcal{C}_x^{\alpha-}(G)$ the space of H\"older continuous functions on $G$ in variable $x$ and  with any H\"older exponent $0<\beta<\alpha$. Similarly, by $\mathcal{C}_x^{(1+\alpha)-}(G)$ we denote the space of differentiable functions $f$ on $G$ such that $\partial_x f=:f_x \in \mathcal{C}_x^{\alpha-}(G)$. 
In what follows, for two sequences of numbers $\{a_n\}$ and $\{b_n\}$, we will write 
$a_n \sim b_n$, if there exists a number $0<c<\infty$ such that $\lim_{n\to\infty}a_n/b_n=c$. Respectively, $a_n\simeq b_n$, if $\lim_{n\to\infty}a_n/b_n=1$.

Let us focus on the case $G=\bR$. First we note that the solution $u$ is infinitely differentiable in $t>0$, and it is easy to show that 
$$
\frac{\partial^n u(t,x)}{\partial t^n}=\sigma \int_{\mathbb{R}} 
\frac{\partial^{n-1}}{\partial t^{n-1}} P(t,x-y)\dif W(y), \quad \bP-\textrm{a.s}, \ n\in\bN.
$$

The regularity properties of $u$ in the spatial component $x$ are more delicate. Towards this end,  we fix $a,b\in\bR$ such that $a>0$ and $b\geq 0$ (or $a\geq 0$ and $b>0$), and denote by $\Delta_h$, for some $h>0$, the difference quotient operator (acting on the spatial component) of the form 
\begin{equation}\label{eq:derApprox}
\Delta_h f(t,x):=\frac{f(t,x+ah)-f(t,x-bh)}{(a+b)h}, \quad t>0, \ x\in\bR,
\end{equation}
which should be viewed as an approximation of $f_x(t,x)$. For the sake of simplicity of writing we will suppress the dependence of $\Delta_h$ on $a,b$, while keeping in mind this dependence which will be important in the next section. 
We will also make use of the notation 
$$
v(t,x):=\sigma\int_0^t \int_{\mathbb{R}}P_x(s,x-y)\dif W(y)\dif s, \quad t>0, \ x\in\bR. 
$$
We start with a key technical result that will be used  throughout. 

\begin{lemma}\label{lemma:uhDuhproducs}
Let $t,t'>0$ and $x,y\in \mathbb{R}$. Then,
\begin{align}
\mathbb{E}\Delta_hu\left(t,x\right)\Delta_hu\left(t',y\right)
&=\frac{\sigma^2}{2\sqrt{\pi \theta}(a+b)^2h^2}
	\int_0^{t'}\int_0^t 
	\bigg(\frac{2e^{-\frac{(x-y)^2}{4\theta (s_1+s_2)}}}{\sqrt{s_1+s_2}} \nonumber\\
& \qquad	-\frac{e^{-\frac{(x-y+(a+b)h)^2}{4\theta (s_1+s_2)}}}{\sqrt{s_1+s_2}}-
	\frac{e^{-\frac{(x-y-(a+b)h)^2}{4\theta (s_1+s_2)}}}{\sqrt{s_1+s_2}}
	\bigg)\dif s_1\dif s_2,  \label{eq:A1}\\ 
\mathbb{E}\Delta_hu\left(t,x\right)v\left(t',x\right) 
& =\frac{\sigma^2}{4\sqrt{\pi}\theta^{3/2}(a+b)}
 \int_0^{t'}\int_0^t \bigg(\frac{ae^{-\frac{a^2h^2}{4\theta(s_1+s_2)}}}{(s_1+s_2)^{3/2}} 
 +\frac{be^{-\frac{b^2h^2}{4\theta(s_1+s_2)}}}{(s_1+s_2)^{3/2}} \bigg) \dif s_1\dif s_2, \label{eq:A2} \\ 
 \mathbb{E}v\left(t,x\right)v\left(t',y\right)& =
 \frac{\sigma^2}{4\sqrt{\pi}\theta^{3/2}}\int_0^{t'}\int_0^t\frac{e^{-\frac{(x-y)^2}{4\theta(s_1+s_2)}}}{(s_1+s_2)^{3/2}} \left(1-\frac{(x-y)^2}{2\theta(s_1+s_2)}\right)\dif s_1\dif s_2.\label{eq:A3}
\end{align}
\end{lemma}
\begin{proof}
The proof is deferred to Appendix~\ref{append:A}.
\end{proof}


\begin{theorem}\label{th:OptimalReg}
For each $t>0$, the function $u(t,\,\cdot\,)$ is $\mathbb{P}-$a.s. differentiable, and its derivative is given by 
\begin{equation}\label{eq:ux}
u_x(t,x) = \sigma\int_0^t \int_{\mathbb{R}}P_x(s,x-y)\dif W(y)\dif s.
\end{equation}
Moreover, $u_x(t,x)\in C^{1/4-,1/2-}_{t,x}([0,T]\times\bR)$. 
\end{theorem}
\begin{proof} 
To prove \eqref{eq:ux}, we will show that, for any $t>0$ and $x\in\bR$,
\begin{equation}\label{eq:duEv}
\lim_{h\to 0} \left[\Delta_hu(t,x)-v(t,x)\right]=0, \quad \mathbb{P}-\textrm{a.s.}
\end{equation}	
From Lemma~\ref{lemma:uhDuhproducs}, we have
$\mathbb{E}\left|\Delta_hu(t,x)-v(t,x)\right|^2 = A_1+2A_2+A_3$, where $A_1,A_2$ and $A_3$ are the right hand sides of \eqref{eq:A1}, \eqref{eq:A2} and \eqref{eq:A3} with $t=t'$ and $x=y$ respectively. 
Applying Lemma~\ref{lemma:intbyparts-1} to $A_1$, $A_2$ and $A_3$, we get
\begin{align*}
A_1 = &\frac{4\sigma^2}{3\sqrt{\theta\pi}(a+b)^2h^2}\left((2t)^{3/2}\left(1-e^{-\frac{(a+b)^2h^2}{8\theta t}}\right)-2t^{3/2} \left(1-e^{-\frac{(a+b)^2h^2}{4\theta t}}\right)\right)\\
&-\frac{4\sigma^2}{3\sqrt{\pi}\theta^2}\left((2t)^{1/2}e^{-\frac{(a+b)^2h^2}{8\theta t}}+2t^{3/2} e^{-\frac{(a+b)^2h^2}{4\theta t}}\right)\\
&+\frac{2(a+b)h\sigma^2}{3\sqrt{\pi}\theta^2} \left[\int_{\frac{4\theta t}{(a+b)^2h^2}}^{\frac{8\theta t}{(a+b)^2h^2}} s^{-3/2}e^{-\frac{1}{s}}\dif s-\int_0^{\frac{4\theta t}{(a+b)^2h^2}} s^{-3/2}e^{-\frac{1}{s}}\dif s \right]\\
&+\frac{(a+b)h\sigma^2}{2\sqrt{\pi}\theta^2}\int_0^{\frac{4\theta t}{(a+b)^2h^2}}\int_0^{\frac{4\theta t}{(a+b)^2h^2}}(s_1+s_2)^{-5/2}e^{-\frac{1}{s_1+s_2}}\dif s_1\dif s_2,
\end{align*}
\begin{align*}
A_2= &\frac{\sigma^2 (2t)^{1/2}}{\sqrt{\pi}\theta^{3/2}(a+b)}\left(ae^{-\frac{a^2h^2}{8\theta t}}+be^{-\frac{b^2h^2}{8\theta t}}\right)-\frac{2\sigma^2 t^{1/2}}{\sqrt{\pi}\theta^{3/2}(a+b)} \left(ae^{-\frac{a^2h^2}{4\theta t}}+be^{-\frac{b^2h^2}{4\theta t}}\right)\\
&-\frac{a^2h\sigma^2}{2\sqrt{\pi}\theta^2(a+b)} \left[\int_{\frac{4\theta t}{a^2h^2}}^{\frac{8\theta t}{a^2h^2}} s^{-3/2}e^{-\frac{1}{s}}\dif s-\int_0^{\frac{4\theta t}{a^2h^2}}  s^{-3/2}e^{-\frac{1}{s}}\dif s \right]\\
&-\frac{b^2h\sigma^2}{2\sqrt{\pi}\theta^2(a+b)} \left[\int_{\frac{4\theta t}{b^2h^2}}^{\frac{8\theta t}{b^2h^2}} s^{-3/2}e^{-\frac{1}{s}}\dif s-\int_0^{\frac{4\theta t}{b^2h^2}} s^{-3/2}e^{-\frac{1}{s}}\dif s \right]\\
&-\frac{a^2h\sigma^2}{4\sqrt{\pi}\theta^2(a+b)}\int_0^{\frac{4\theta t}{a^2h^2}}\int_0^{\frac{4\theta t}{a^2h^2}} (s_1+s_2)^{-5/2}e^{-\frac{1}{s_1+s_2}}\dif s_1\dif s_2\\
&-\frac{b^2h\sigma^2}{4\sqrt{\pi}\theta^2(a+b)}\int_0^{\frac{4\theta t}{b^2h^2}}\int_0^{\frac{4\theta t}{b^2h^2}} (s_1+s_2)^{-5/2}e^{-\frac{1}{s_1+s_2}}\dif s_1\dif s_2,
\end{align*}
and
$$
A_3 = -\frac{\sigma^2}{\theta^{3/2}\sqrt{\pi}}(2t)^{1/2} + \frac{2\sigma^2}{\theta^{3/2}\sqrt{\pi}}t^{1/2}.
$$
Combining the above, and taking into account that 
$$
\int_0^{\infty} s^{-3/2}e^{-1/s}\dif s=\sqrt{\pi},   \quad
\int_0^{\infty}\int_0^{\infty} (s_1+s_2)^{-5/2}e^{-\frac{1}{s_1+s_2}}\dif s_1\dif s_2=\sqrt{\pi},
$$
we get at once that 
\[
\lim_{h\to 0} \frac{\mathbb{E}\left|\Delta_hu(t,x)-v(t,x)\right|^2}{h} =\frac{(a^2-ab+b^2)\sigma^2}{3(a+b)\theta^2}. 
\]
Note that $a^2-ab+b^2>0$, and thus, for small $h>0$, we have that 
$$
\mathbb{E}\left|\Delta_hu(t,x)-v(t,x)\right|^2\leq Ch,\quad \mbox{for some}\ C:=C(\theta,\sigma,a,b)>0.
$$
Hence, \eqref{eq:duEv} follows by Kolmogorov's continuity theorem, and thus \eqref{eq:ux} is proved.  

Using \eqref{eq:ux} and by applying Lemma~\ref{lemma:uhDuhproducs}, for $t,h>0$ and $x\in \mathbb{R}$, we have
\begin{align*}
\mathbb{E}\big|u_x(t,x+h) & -u_x(t,x)\big|^2 
 = 	\frac{\sigma^2}{2\theta^{3/2}\sqrt{\pi}}\int_0^t\int_0^t 
	\frac{1-e^{-\frac{h^2}{4\theta(s_1+s_2)}}}{(s_1+s_2)^{3/2}}\dif s_1\dif s_2 \\
& \qquad +\frac{\sigma^2h^2}{4\theta^{5/2}\sqrt{\pi}} \int_0^{t}\int_0^{t} \frac{e^{-\frac{h^2}{4\theta(s_1+s_2)}}}{(s_1+s_2)^{5/2}} 	\dif s_1\dif s_2\\
& \leq C_1h^{2\varepsilon} \int_0^t\int_0^t \frac{1}{(s_1+s_2)^{3/2+\varepsilon}} \dif s_1\dif s_2
+ C_2h\int_0^{t/h^2}\int_0^{t/h^2} \frac{e^{-\frac{1}{4\theta(s_1+s_2)}}}{(s_1+s_2)^{5/2}} 
\dif s_1\dif s_2 \\ 
& \leq C_3h^{2\varepsilon},
\end{align*}
for some constants $C_1,C_2,C_3>0$ and $0<\varepsilon<1/2$. 
By Kolmogorov's continuity theorem, we conclude that 	$u_x(t,\cdot)$ has a continuous modification that is  H\"older continuous with any exponent $\alpha<1/2$.

On the other hand, for each $x\in \mathbb{R}$, and $0\leq s\leq t$ such that $|t-s|\leq1$, we have
\begin{align*}
\mathbb{E}|u_x(t,x)-u_x(s,x)|^2 & =
	\frac{\sigma^2}{4\sqrt{\pi}\theta^{3/2}}\int_s^t\int_s^t (r_1+r_2)^{-3/2}\dif r_1 \dif r_2\nonumber\\
& = \frac{\sigma^2}{\sqrt{\pi}\theta^{3/2}} \left[-(2t)^{1/2}+(t+s)^{1/2}+(t+s)^{1/2}-(2s)^{1/2}\right] \\
& \leq C |t-s|^{1/2}. \label{eq:uxt-s}
\end{align*}
Again, in view of Kolmogorov's continuity theorem  $u_x(\cdot, x)$ has a continuous modification that is H\"older continuous with exponent $\alpha<1/4$ on $[0,T]$, for any $T>0$.
The proof is complete. 
\end{proof}

The case of bounded domain was studied in \cite{KimLototsky2017}, where the authors showed that Theorem~\ref{th:OptimalReg} holds true for $G=(0,\pi)$ and Dirichlet or Neumann boundary conditions. 
As we will show in the next section, the space regularity of $u_x(t,x)$ is optimal - a fact critically used in deriving the estimators for $\theta$ and $\sigma$.

\section{Statistical analysis}\label{sec:stats}

In this section we will derive consistent and asymptotically normal estimators for the parameters $\theta$ and $\sigma$, assuming discrete sampling of the solution in spatial component (or spatial and time components). Unless otherwise mentioned, we take $G=\bR$. 

Consider the uniform partition $\{x_i:=A+(B-A)i/N,\ i=0,...,N\}$ of the interval  $[A,B]\subset\bR$. We recall that the quadratic variation\footnote{For convenience, we consider the quadratic variation only with respect to uniform partitions.} of process $X$ on $[A,B]$, is defined as 
$$
V(X;[A,B]):=\lim_{N\to \infty} \sum_{i=1}^N |X(x_i)-X(x_{i-1})|^2, \quad \bP-\mbox{a.s.}
$$

\subsection{Discrete observations of $u_x$}\label{sec:discrtDerivU}

In view of Theorem~\ref{th:OptimalReg},  $u_x(t,\cdot)\in \mathcal{C}_x^{1/2-}(\mathbb{R})$, for any $t>0$, and thus one would expect that the quadratic variation of $u_x(t,\cdot)$ on any finite interval $[A,B]\subset\bR$, is finite ($\bP$-a.s. or in probability). Indeed, using \eqref{eq:ux}, applying integration by parts combined with some simple properties of the heat kernel, we deduce 
\begin{align*}
u_x(t,x) & = -\sigma \int_0^t \int_{\mathbb{R}}P_{xy}(t-s,x-y)W(y)\dif y\dif s 
=\sigma \int_0^t \int_{\mathbb{R}} P_{xx}(t-s,x-y)W(y)\dif y\dif s \\
& =\frac{\sigma}{\theta}\int_0^t\int_{\mathbb{R}}P_{t}(t-s,x-y)W(y)\dif y\dif s
=-\frac{\sigma}{\theta} \int_0^t \int_{\mathbb{R}}P_{s}(t-s,x-y)W(y)\dif y\dif s\\ & =-\frac{\sigma}{\theta}\int_{\mathbb{R}}\int_0^t P_{s}(t-s,x-y)\dif s W(y)\dif y
=-\frac{\sigma}{\theta}\int_{\mathbb{R}}\lim_{t'\to t}\int_0^{t'}P_{s}(t-s,x-y)W(y)\dif s\dif y,
\end{align*}
which consequently implies that 
\begin{align}
u_x(t,x)&=-\frac{\sigma}{\theta}\int_{\mathbb{R}} \lim_{t'\to t} \left(P\left(t-t',x-y\right) W(y)\right)\dif y + \frac{\sigma}{\theta}\int_{\mathbb{R}}P(t,x-y)W(y)\dif y \\
& =-\frac{\sigma}{\theta}W(x)+\frac{\sigma}{\theta}\int_{\mathbb{R}}P(t,x-y)W(y)\dif y. \label{eq:WR}
\end{align}
Note that $\int_{\mathbb{R}} P(t,x-y)W(y)\dif y$ is  infinitely differentiable 
in $x$, and thus in view of \cite[Proposition~2.1]{CialencoHuang2017}, we have, for any fixed $t>0$,
\begin{equation}\label{eq:QV-ux}
V\left(u_x(t,\cdot);[A,B]\right)=\frac{\sigma^2}{\theta^2}(B-A), \quad \bP-\mbox{a.s.}
\end{equation}
for any $A,B\in\bR, \ A<B$. 

\begin{remark}\label{rem:optimalReg}
Note that the representation \eqref{eq:WR} implies that the H\"older continuity $3/2-$ of $u(t,\cdot)$ given by Theorem~\ref{th:OptimalReg} is optimal. 
\end{remark}

As a direct consequence of \eqref{eq:QV-ux}, similarly to the method proposed in \cite{CialencoHuang2017}, for any fixed $t>0$, we consider the following estimators for $\theta^2$, assuming $\sigma$ is known, and correspondingly, for $\sigma^2$ assuming that $\theta$ is known,
\begin{align*}
\bar \theta^2_N  := \frac{\sigma^2(B-A)}{\sum_{i=1}^N \left(u_x(t,x_i)-u_x(t,x_{i-1})\right)^2}, \qquad 
\bar \sigma^2_N  :=\frac{\theta^2 \sum_{i=1}^N \left(u_x(t,x_i)-u_x(t,x_{i-1})\right)^2}{B-A}. 
\end{align*}
Consistency and asymptotic normality, as $N\to\infty$, are readily available due to \cite[Example~2.1]{CialencoHuang2017}. By extension, assuming that $u_x(t,x_i)$ is also sampled at some discrete time points, say $0<t_1<\cdots < t_M=T$, for some fixed $T>0$ and $M\in\bN$, we consider the estimators 
\begin{align}
\wh \theta^2_{N,M}&:= \frac{\sigma^2(B-A)M} {\sum_{j=1}^M\sum_{i=1}^N \left(u_x(t_j,x_i)-u_x(t_j,x_{i-1})\right)^2},  \label{eq:t-estimator-theta}\\
\wh \sigma^2_{N,M}&:=\frac{\theta^2 \sum_{j=1}^M\sum_{i=1}^N \left(u_x(t_j,x_i)-u_x(t_j,x_{i-1})\right)^2}{(B-A)M}, \label{eq:t-estimator-sigma}
\end{align}
which can be viewed as `time-average' counterparts of $\bar\theta_{N}^2$, and $\bar\sigma_{N}^2$.  Strong consistency of $\wh\theta_{N,M}^2$ and $\wh \sigma_{N,M}^2$, as $N\to\infty$, follow trivially from the strong consistency of $\bar\theta_{N}^2$ and $\bar \sigma_{N}^2$. On the other hand, the proof of asymptotic normality is more delicate and technically involved. 
We use elements of Malliavin calculus and Stein's method to prove asymptotic normality of $\wh\theta_{N,M}^2$ and $\wh \sigma_{N,M}^2$, as $N\to\infty$, which is performed over the course of several technical results listed below, while most of the proofs are postponed to Appendix~\ref{append:ProofsStats}.

In what follows, we will denote by $\Upsilon_{N,M}$ the space-time sampling grid, namely 
\begin{align*}
\Upsilon_{N,M}= \Big\{ (t_k,x_i) \mid & 0<t_1<\cdots<t_M=T, \  A=x_0<x_1<\cdots<x_N=B, \\ 
	& \quad x_{j+1} - x_j = x_1-x_0, \ j=1,\ldots,N-1 \Big\},
\end{align*}
for some fixed $T>0$, $A,B\in\bR$. In addition, we set $\tau:= \sigma^2/(\theta^{3/2}\sqrt{\pi})$.

\begin{proposition}\label{prop:uxasymp}
For every $t>0$, $x\in \mathbb{R}$ and $\alpha<1$, we have
$$
\lim_{N\to \infty}N^{\alpha}\left[N\mathbb{E} \left|u_x\left(t,x+\frac{B-A}{N}\right)-u_x\left(t,x\right)\right|^2
-\frac{\sigma^2(B-A)}{\theta^{2}}\right]=0.
$$
In particular, $$
\lim_{N\to \infty}N\mathbb{E} \left|u_x\left(t,x+\frac{B-A}{N}\right)-u_x\left(t,x\right)\right|^2
=\frac{\sigma^2(B-A)}{\theta^{2}}.
$$
\end{proposition}
\begin{proof} 
The proof is deferred to Appendix~\ref{append:ProofsStats}.
\end{proof}

We next investigate the limit distribution of the centered and averaged spatial quadratic variation of $u_x$ defined as 
\begin{equation*}
Q_{N,M}:=\frac{1}{M}\sum_{j=1}^M\sum_{i=1}^N \left[\frac{\left(u_x(t_j,x_i)-u_x(t_j,x_{i-1})\right)^2}
{\mathbb{E}\left(u_x(t_j,x_i)-u_x(t_j,x_{i-1})\right)^2}-1\right]=
\frac{1}{M}\sum_{j=1}^M\sum_{i=1}^N \left[\frac{U(t_j,x_i)}{\mathbb{E}U(t_j,x_i)}-1\right],
\end{equation*}
where $U(t,x_i):=\left(u_x(t,x_i)-u_x(t,x_{i-1})\right)^2$.
In view of Proposition~\ref{prop:uxasymp}, $\mathbb{E} [U(t,x_i)]$ is independent of $i$, and for simplicity we will write $\mathbb{E}U(t,x_i)=\mathbb{E} U(t)$.

Note that for any $j=1,\dots, M$ and $k=1,\dots N$, 
\begin{align}\label{eq:innprod-tjxk}
u_x(t_j,x_k)-u_x(t_j,x_{k-1})=\int_{x_{k-1}}^{x_k} \dif u_x(t_j,\cdot)=\int_{\mathbb{R}}\int_0^{\infty}1_{[x_{k-1},x_k]}(y)\delta(t-t_j)\dif u_x(t,y)
\end{align}
is centered and stationary Gaussian. 

Let us denote by $\mathcal{H}$ the canonical Hilbert space associated to the Gaussian sequence with unit variance $\left\{\displaystyle\frac{u_x(t_j,x_k)-u_x(t_j,x_{k-1})}{\left(\mathbb{E}U(t_j)\right)^{1/2}},\ j=1,\dots, M,\ k=1,\dots,N\right\}_{M,N\in \mathbb{N}}$ \cite[Proposition 7.2.3]{NourdinPeccati2012}. 
By the product formula \cite[Theorem 2.7.7]{NourdinPeccati2012} or \cite[Proposition 1.1.3]{Nualart2006} for multiple stochastic integrals, we rewrite $Q_{N,M}$ as
\begin{equation}\label{eq:QN-multi}
Q_{N,M}=\frac{1}{M}\sum_{j=1}^M \sum_{i=1}^N I_2^{t_j,t_j}\left(1_{[x_{i-1},x_i]}^{\otimes 2}\right),
\end{equation}
where $I_n^{t_1,\dots,t_n}$ is the $n-$th multiple stochastic integral with $n\leq M,N$:
\begin{align*}
I_n^{t_1,\dots,t_n}(f)&:=\int_{\mathbb{R}^n}\int_{(0,\infty)^n} f(y_1,\dots,y_n)\delta(s_1-t_1)\cdots \delta(s_n-t_n) \dif u_x(s_1,y_1)\cdots \dif u_x(s_n,x_n)\\
&=\int_{\mathbb{R}^n}f(y_1,\dots,y_n) \dif u_x(t_1,y_1)\cdots \dif u_x(t_n,y_n),
\end{align*}
for a symmetric function $f$ and $\otimes$ denoting the tensor product.
Accordingly, using the isometry for  multiple integrals, we obtain 
\begin{align*}
\mathbb{E}Q_{N,M}^2 & =\frac{1}{M^2}\sum_{k,\ell =1}^M
\sum_{i,j=1}^N \mathbb{E} \left[I_2^{t_k,t_k}\left(1_{[x_{i-1},x_i]}^{\otimes 2}\right)
I_2^{t_{\ell},t_{\ell}}\left(1_{[x_{j-1},x_j]}^{\otimes 2}\right)\right]  \\
& =\frac{2}{M^2}\sum_{k,\ell =1}^M
\frac{1}{\mathbb{E} U(t_k)\mathbb{E} U(t_{\ell})} \sum_{i,j=1}^N \big[\mathbb{E}\left(u_x(t_k,x_i)-u_x(t_k,x_{i-1})\right)\left(u_x(t_{\ell},x_j)-u_x(t_{\ell},x_{j-1})\right)\big]^2 \\  
&=\frac{2}{M^2}\sum_{k=1}^M\left(\mathbb{E} U(t_k)\right)^{-2} 
\sum_{i=1}^N \big[\mathbb{E}\left(u_x(t_k,x_i)-u_x(t_k,x_{i-1})\right)\left(u_x(t_{\ell},x_i)-u_x(t_{\ell},x_{i-1})\right)\big]^2 \\
&+\frac{4}{M^2}\sum_{\ell>k}
	\frac{1}{\mathbb{E} U(t_k)\mathbb{E} U(t_{\ell})} \sum_{i=1}^N \big[\mathbb{E}\left(u_x(t_k,x_i)-u_x(t_k,x_{i-1})\right)\left(u_x(t_{\ell},x_i)-u_x(t_{\ell},x_{i-1})\right)\big]^2\\
&+\frac{4}{M^2}\sum_{k=1}^M\left(\mathbb{E} U(t_k)\right)^{-2} \sum_{i>j} \big[\mathbb{E}\left(u_x(t_k,x_i)-u_x(t_k,x_{i-1})\right)\left(u_x(t_{k},x_j)-u_x(t_{k},x_{j-1})\right)\big]^2 \\ 
& +
\frac{8}{M^2}\sum_{\ell>k}\big(\mathbb{E} U(t_k)\mathbb{E}  U(t_{\ell})\big)^{-1} 
\sum_{i>j} \big[\mathbb{E}\left(u_x(t_k,x_i)-u_x(t_k,x_{i-1})\right)\left(u_x(t_{\ell},x_j)-u_x(t_{\ell},x_{j-1})\right)\big]^2 \\
&=:Q_D +Q_{1}+Q_{2}+Q_{ND}.
\end{align*}
By construction, clearly the `diagonal term' $Q_D$ is equal to $2N/M$. 
For $Q_{1}$, we see that 
	\begin{align}\label{eq:limQC1}
	\lim_{N\to \infty} \frac{1}{N}Q_{C1}=2-\frac{2}{M}.
	\end{align}
This follows from the fact that
$$
\big[\mathbb{E}\left(u_x(t_k,x_i)-u_x(t_k,x_{i-1})\right)\left(u_x(t_{\ell},x_i)-u_x(t_{\ell},x_{i-1})\right)\big]^2\simeq \mathbb{E} U(t_k)\mathbb{E} U(t_{\ell})\quad \mbox{as $N\to \infty$}.
$$
To deal with the second cross term $Q_{2}$ and the non-diagonal term $Q_{ND}$, we compute explicitly (in Appendix~\ref{append:ProofsStats}) the expectations in inner sums, that consequently lead to the following representations
\begin{align}
Q_{2} & =\frac{4}{M^2}\sum_{k=1}^M\left(\mathbb{E} U(t_k)\right)^{-2}
\sum_{i=1}^{N-1} (N-i)\Phi_N^2(i,t_{k},t_k),
\label{eq:QC2-mainText}\\
Q_{ND} &=\frac{8}{M^2}\sum_{\ell>k}\big(\mathbb{E} U(t_k)\mathbb{E} U(t_{\ell})\big)^{-1} 
\sum_{i=1}^{N-1} (N-i)\Phi_N^2(i,t_{\ell},t_k),\label{eq:QND-mainText}
\end{align}
where the function $\Phi$ is defined in \eqref{eq:PhiN}.

\begin{proposition}\label{QNDAsym}
The following limits hold true:
\begin{align}
&\lim_{N\to \infty}\frac{1}{N}Q_{2}=0,   \label{eq:limQC2}\\ 
&\lim_{N\to \infty}\frac{1}{N}Q_{ND}=0,   \label{eq:limQND}\\ 
&\lim_{N\to \infty}\frac{1}{N}\mathbb{E}Q_{N,M}^2 =2. \label{eq:limEQ2NM}
\end{align}
\end{proposition}
\begin{proof}
We note that for any $x,c,c_1,c_2>0$, we have 
$$
\varphi_{c}(0,x) = \partial_z\varphi_{c}(0,x) =\phi_{(c_1,c_2)}(0,x) =\partial_z\phi_{(c_1,c_2)}(0,x) =0,
$$
where $\phi_c$ and $\phi_{(c_1,c_2)}$ are defined in \eqref{eq:phic} and \eqref{eq:phic1c2}.
Moreover, for $x\geq z> 0$, small $z$ and $c_1>c_2$, there exists a constant $C>0$, such that
\begin{align}\label{phi-taylor}
\left|\varphi_{c}(z,x)\right|&\leq  Cz^2\left(c^2(x+z)^2+c\right)e^{-c(x-z)^2}, \nonumber\\
\left|\phi_{(c_1,c_2)}(z,x)\right|& \leq\frac{C}{\sqrt{c_2}}z^2\left(\frac{1}{c_2}(x+z)^2+1\right)e^{-\frac{(x-z)^2}{c_1}},
\end{align}
which can be obtained, for example, by considering Taylor expansion of $\varphi_c$ and $\phi_{(c_1,c_2)}$ in $z$ with the remainder of order two.
Therefore, for each $i=1,2,\dots,N-1$ and $k,\ell=1,2,\dots,M$, we get
\begin{equation}\label{Phi-taylor-1}
\left|\Phi_N(i,t_{\ell},t_k)\right|\leq \frac{C}{\sqrt{t_1}N^{2}}
\left(\frac{(i+1)^2}{t_1N^{2}}+1\right).
\end{equation}
Hence, by \eqref{eq:QC2-mainText} and \eqref{eq:QND-mainText}, we get that 
$Q_{2}/N\leq C/N\to 0$ and $Q_{ND}/N\leq C/N\to 0$, as $N\to \infty$, and thus \eqref{eq:limQC2} and \eqref{eq:limQND} are proved. Using them, and recalling that $\bE Q^2_{N,M}=Q_D +Q_{1}+Q_{2}+Q_{ND}$, we get \eqref{eq:limEQ2NM}. This completes the proof. 
\end{proof}

For convenience, we define the normalized version of $Q_{N,M}$ as 
$$
\mathcal{Q}_{N,M}:=\sqrt{\frac{1}{2N}}Q_{N,M}.
$$
Now, we are ready to study the asymptotic normality of $\mathcal{Q}_{N,M}$ by studying the total variation distance between $\mathcal{Q}_{N,M}$ and a standard Gaussian random variable. 

\begin{theorem}\label{th:CLT}
For each $N,M\geq 1$, there exists a positive constant $C$ such that
$$
d_{TV}\left(\mathcal{Q}_{N,M}, \mathcal{N}(0,1)\right)\leq \frac{C}{N},
$$
where $d_{TV}$ denotes the total variation distance and $\cN(0,1)$ is a standard Gaussian random variable. 
\end{theorem}
\begin{proof}
The detailed proof is deferred to Appendix~\ref{append:ProofsStats}.	
\end{proof}

Finally, we present the main result of this section. 

\begin{theorem}\label{th:est-ux-t} 
Assume that $u_x(t,x)$ is observed at the grid points $\Upsilon_{N,M}$. 
Then:
\begin{enumerate}[(i)]
\item given that $\sigma$ is known,\footnote{$\w\lim$ stands for convergence in distribution of random variables. }
\begin{align}
\lim_{N\to \infty} \wh\theta^2_{N,M} &=\theta^2, \quad \bP-\textrm{a.s.},\label{eq:tildeThetaConsist} \\
\w\lim_{N\to \infty} \sqrt{N}\left(\wh\theta^2_{N,M}-\theta^2\right)&=
\mathcal{N}\left(0,2\theta^4\right). \label{eq:tildeThetaAsNor}
\end{align}

\item similarly, given that $\theta$ is known,
\begin{align}
\lim_{N\to \infty} \wh\sigma^2_{N,M} & =\sigma^2, \quad \bP-\textrm{a.s.},\label{eq:tildeSigmaConsist} \\
\w\lim_{N\to \infty} \sqrt{N}\left(\wh\sigma^2_{N,M}-\sigma^2\right) & = 
\mathcal{N}\left(0,2\sigma^4\right). \label{eq:tildeSigmaAnNor}
\end{align}
\end{enumerate}
\end{theorem}
\begin{proof}
Strong consistency \eqref{eq:tildeThetaConsist} and \eqref{eq:tildeSigmaConsist} are already proved. 
Next we will prove \eqref{eq:tildeSigmaAnNor} - asymptotic normality of $\wh\sigma^2_{N,M}$. 
Aiming to connect $\cQ_{N,M}$ with $\wh\sigma_{N,M}^2$, we rewrite $\mathcal{Q}_{N,M}$ as 
\begin{align*}
\mathcal{Q}_{N,M} &=\frac{1}{M\sqrt{2N}}\sum_{j=1}^M\sum_{i=1}^N \left[\frac{\left(u_x(t_j,x_i)-u_x(t_j,x_{i-1})\right)^2}
	{\mathbb{E}\left(u_x(t_j,x_i)-u_x(t_j,x_{i-1})\right)^2}-1\right]\\
&=\frac{1}{M\sqrt{2}}\sum_{j=1}^M\Bigg[\frac{\left(\sum_{i=1}^N 
		\left(u_x(t_j,x_i)-u_x(t_j,x_{i-1})\right)^2\right)-\frac{\sigma^2(B-A)}{\theta^2}}{\sqrt{N}\mathbb{E}\left(u_x(t_j,x_1)-u_x(t_j,x_{0})\right)^2}\\
	&\qquad \qquad \qquad \qquad \qquad \qquad \qquad +\frac{\frac{\sigma^2(B-A)}{\theta^2}-N\mathbb{E}\left(u_x(t_j,x_i)-u_x(t_j,x_{i-1})\right)^2}{\sqrt{N}\mathbb{E}\left(u_x(t_j,x_1)-u_x(t_j,x_0)\right)^2}\Bigg].
\end{align*}
Then, from Proposition \ref{prop:uxasymp}, we have
\begin{align*}
\mathcal{Q}_{N,M}&\overset{a.s.}\simeq \frac{\sqrt{N}\theta^2}{M\sqrt{2}\sigma^2(B-A)}\left[\sum_{j=1}^M\sum_{i=1}^N 
	\left(u_x(t_j,x_i)-u_x(t_j,x_{i-1})\right)^2-\frac{M\sigma^2(B-A)}{\theta^2}\right]\\
	&+ \frac{\sqrt{N}\theta^2}{M\sqrt{2}\sigma^2(B-A)}\sum_{j=1}^M
	\left[\frac{\sigma^2(B-A)}{\theta^2}-N\mathbb{E}\left(u_x(t_j,x_i)-u_x(t_j,x_{i-1})\right)^2
	\right],\quad \mbox{as}\ N\to \infty.
\end{align*}
Equivalently, as $N\to \infty$,
\begin{align*}
\mathcal{Q}_{N,M}&\overset{a.s.}\simeq
\sqrt{\frac{N}{2\sigma^4}}\left(\widehat{\sigma}^2_{N,M}-\sigma^2\right)\\
&+\frac{\sqrt{N}\theta^2}{M\sqrt{2}\sigma^2(B-A)}\sum_{j=1}^M
\left[\frac{\sigma^2(B-A)}{\theta^2}-N\mathbb{E}\left(u_x(t_j,x_i)-u_x(t_j,x_{i-1})\right)^2\right].
\end{align*}
By Proposition \ref{prop:uxasymp} again, we have
\begin{equation*}
\sqrt{N}\left[\frac{\sigma^2(B-A)}{\theta^2}-N\mathbb{E}\left(u_x\left(t,x_i\right)-u_x(t,x_{i-1})\right)^2	\right]
\longrightarrow 0,\quad \mbox{as}\ N \to \infty.
\end{equation*}
Moreover, Theorem \ref{th:CLT} combined with Theorem \ref{MSthm2} implies that $\displaystyle\w\lim_{N\to \infty}\mathcal{Q}_{N,M}=\mathcal{N}(0,1)$. 
This, and using Slutsky's theorem, implies \eqref{eq:tildeSigmaAnNor}. 
Finally, \eqref{eq:tildeThetaAsNor} follows by the Delta method and \eqref{eq:tildeSigmaAnNor}. 
The proof is complete. 
\end{proof}

\subsection{Discrete observations of $u$}\label{sec:discrtSolU}

While the estimators $\wh\theta_{N,M}^2$ and $\wh \sigma_{N,M}^2$  have desirable statistical properties,  they assume that the observer measures the values of the derivative $u_x(t,x)$ which practically speaking is a questionable assumption. As argued in  \cite{CialencoKimLototsky2019}, for similar estimators involving spatial derivatives, it is natural to replace  $u_x(t_j,x_i)$ in \eqref{eq:t-estimator-theta} and \eqref{eq:t-estimator-sigma} by a finite difference approximation of derivative such as $\Delta_h u(t_j,x_i)$, defined in \eqref{eq:derApprox}. In particular, with $h=(B-A)/N$ and $a=1,b=0$ (or $a=0,b=1$), assuming the solution $u$ is sampled discretely in the physical space-time domain $\Upsilon_{N,M}$, we can consider the following estimators 
\begin{align}
\check{\theta}^2_{N,M} & := \frac{\sigma^2M(B-A)^3}{N^2\sum_{j=1}^M\sum_{i=1}^{N} \left( u(t_j,x_{i+1})-2u(t_j,x_i)+u(t_j,x_{i-1})\right)^2}, \label{eq:Dis-estimators-theta-forward}\\
\check{\sigma}^2_{N,M} & := \frac{\theta^2 N^2 \sum_{j=1}^M\sum_{i=1}^{N} \left(u(t_j,x_{i+1})-2u(t_j,x_i)+u(t_j,x_{i-1})\right)^2}{M(B-A)^3},
\label{eq:Dis-estimators-sigma-forward}
\end{align}
where, for the convenience of writing,  we use the same notation $\Upsilon_{N,M}$ to denote the time-space grid $\Upsilon_{N,M}$ from previous section that contains the additional points $(t_j,x_{N+1}) = (t_j,B+h)$.  
Since $\check{\theta}^2_{N,M}$ and $\check{\sigma}^2_{N,M}$ are just approximations of $\wh{\theta}^2_{N,M}$ and $\wh{\sigma}^2_{N,M}$, it is tempting to argue that the consistency and asymptotic normality will be preserved. However, as we prove below, since the finite difference approximation of the derivative is done at the critical spatial regularity, namely $u(t,x)$ is only $3/2-$ H\"older continuous, 
$\check{\theta}^2_{N,M}$ and $\check{\sigma}^2_{N,M}$ are no longer consistent, or asymptotic normal, unless a multiplicative adjustment is made. This `bias' was first noticed in the numerical experiments in \cite{CialencoKimLototsky2019}. Here, we formally derive an explicit value of this bias. In particular, we prove that the correction constant depends on the finite difference used for approximations. 

Towards this end, we fix $T>0$, $M\in\bN$,  $\gamma\geq1$, and assume that $a,b$ from \eqref{eq:derApprox}  are such that $a,b\in[0,1], \ a+b>0$.  We denote by $\widetilde{\Upsilon}_{N,M}$ the space-time sampling grid of the form 
\begin{align*}
\widetilde{\Upsilon}_{N,M}= \Big\{ (t_k,y_i),(t_k,z_i) \mid & \ 0<t_1<\cdots<t_M=T,  \ x_{j+1} - x_j = x_1-x_0,\\
& y_j:=x_j+\frac{a(B-A)}{N^{\gamma}}, 
 \ z_j:=x_j-\frac{b(B-A)}{N^{\gamma}}, 
\ j=0,\ldots,N \Big\}.
\end{align*}
See the picture below corresponding to a time cross section of $\widetilde{\Upsilon}_{N,M}$.  

\begin{center}
\begin{picture}(300,60)
	\put(20,20){\line(0,1){10}}
	\put(18,15){\tiny $x_0$}
	
	\put(55,20){\line(0,1){6}}
	\put(51,15){\tiny $x_1$}
	
	\put(90,20){\line(0,1){6}}
	\put(86,15){\tiny $x_2$}
	\put(125,20){\line(0,1){6}}
	\put(121,15){\tiny $x_3$}
	\put(152,15){\tiny $\cdots \cdots$}
	\put(200,20){\line(0,1){6}}
	\put(194,15){\tiny $x_{N-1}$}
	\put(235,20){\line(0,1){6}}
	\put(229,15){\tiny $x_{N}$}
	\put(0,25){\line(1,0){250}}
	
	\put(12,23){\line(0,1){8}}
	\put(10,35){\tiny $z_0$}	

	\put(26,23){\line(0,1){8}}
	\put(22,35){\tiny $y_0$}	

	\put(47,23){\line(0,1){8}}
	\put(43,35){\tiny $z_1$}
	
	\put(82,23){\line(0,1){8}}
	\put(78,35){\tiny $z_2$}
	\put(117,23){\line(0,1){8}}
	\put(114,35){\tiny $z_3$}
	\put(192,23){\line(0,1){8}}
	\put(182,35){\tiny $z_{N-1}$}	
	\put(227,23){\line(0,1){8}}
	\put(217,35){\tiny $z_{N}$}
	
	\put(61,23){\line(0,1){8}}
	\put(58,35){\tiny $y_1$}
	\put(96,23){\line(0,1){8}}
	\put(93,35){\tiny $y_2$}
	\put(131,23){\line(0,1){8}}
	\put(128,35){\tiny $y_3$}
	\put(206,23){\line(0,1){8}}
	\put(206,44){\vector(0,-1){12}}
	\put(199,47){\tiny $y_{N-1}$}
	\put(241,23){\line(0,1){8}}
	\put(241,44){\vector(0,-1){12}}	
	\put(234,47){\tiny $y_{N}$}
\end{picture}
\end{center}
When $\gamma = 1$ and $a=1,b=0$, we recover the uniform spatial grid, i.e. $\widetilde{\Upsilon}_{N,M}=\Upsilon_{N,M}$. We also obtain a uniform grid when $a+b=1$ or $a+b=1/2$ or $a=b=1$, all being equal to $\Upsilon_{N,M}$ up to some spatial end points. Generally speaking $\widetilde{\Upsilon}_{N,M}$ is not necessarily  a uniform grid, and if $\gamma>1$, then, as $N$ increases, the distance between grid points $x_j,z_j$ (or $y_j,x_j$) decays faster than the reciprocal of number of sampled points.  
The uniform sampling grid $\Upsilon_{N,M}$ is of the most practical interest, while the proposed general sampling scheme is mostly of theoretical value that shows that the design of the sampling scheme may chance the asymptotic properties of the estimators. 

Consider the following estimators
\begin{align}
\widetilde{\theta}^2_{N,M} & := \frac{\sigma^2M(B-A)}{\sum_{j=1}^M\sum_{i=1}^{N} \left(\Delta_{\frac{B-A}{N^{\gamma}}}  u(t_j,x_i)-\Delta_{\frac{B-A}{N^{\gamma}}}u(t_j,x_{i-1})\right)^2}, \label{eq:Dis-estimators-theta}\\
\widetilde{\sigma}^2_{N,M} & := \frac{\theta^2 \sum_{j=1}^M\sum_{i=1}^{N} \left(\Delta_{\frac{B-A}{N^{\gamma}}}u(t_j,x_i)-\Delta_{\frac{B-A}{N^{\gamma}}}u(t_j,x_{i-1})\right)^2}{M(B-A)}.
\label{eq:Dis-estimators-sigma}
\end{align}
In addition, we put
\begin{align*}
	\mu:=\left\{
	\begin{array}{lll}
		&\displaystyle\frac{1}{a+b}-\frac{1}{3(a+b)^2}, &\mbox{if}\ \gamma=1,\ a+b\geq 1\\
		&1-\displaystyle\frac{(a+b)}{3}, &\mbox{if}\ \gamma=1,\ a+b< 1\\
		&1, & \mbox{if} \ \gamma>1,
	\end{array}
	\right.
\end{align*}
that will serve as the correction constant or it reciprocal. 
Note that with $\gamma=1$, $a=1,b=0$, we have $\widetilde{\theta}^2_{N,M} = \check{\theta}^2_{N,M}$ and $\widetilde{\sigma}^2_{N,M} = \check{\sigma}^2_{N,M}$, and in this case $\mu=2/3$. Moreover, $\mu=2/3$, when $\gamma=1$, and $a+b=1$ or $a+b=1/2$, both corresponding to sampling the solution also at a uniform spatial grid. On the other hand, when $\gamma=1$ and $a=b=1$, we get $\mu=5/12$. 
Recall that these are the only uniform spatial grids when $\gamma=1$.

\begin{theorem}\label{th:est-Du}
	Assume that the solution $u(t,x)$ is observed at the mesh points $\widetilde{\Upsilon}_{N,M}$. 
	In addition, assume that $a,b\in[0,1], \ a+b>0$ in \eqref{eq:derApprox}. 
	Then:
	\begin{enumerate}[(i)]
		\item given that $\sigma$ is known,
		\begin{align}
			&\bP\!-\!\lim_{N\to \infty}\mu\widetilde{\theta}^2_{N,M} =\theta^2, \label{eq:consistTheta2}\\
			&\w\lim_{N\to \infty} \sqrt{N}\left(\mu\widetilde{\theta}^2_{N,M}-\theta^2\right)=
			\mathcal{N}\left(0,2\theta^4\right), \label{eq:AsympNormTheta2}
		\end{align}
	where $\bP\!-\!\lim$ denotes the limit in probability.

		\item given that $\theta$ is known,
		\begin{align}
			&\bP\!-\!\lim_{N\to \infty}\frac{1}{\mu}\widetilde{\sigma}^2_{N,M} =\sigma^2, \label{eq:consistSigma2}\\
			&\w\lim_{N\to \infty} \sqrt{N}\left(\frac{1}{\mu}\widetilde{\sigma}^2_{N,M}-\sigma^2\right)=
			\mathcal{N}\left(0,2\sigma^4\right).\label{eq:AsympNormSigma2}
		\end{align}
	\end{enumerate}
\end{theorem}

\begin{proof}
For every $t>0$ and $x\in \mathbb{R}$, we have
\begin{equation}\label{eq:AsympDu}
\lim_{N\to \infty}N\mathbb{E}
\left|\Delta_{\frac{B-A}{N^{\gamma}}}u\left(t,x+ (B-A)/N\right) - \Delta_{\frac{B-A}{N^{\gamma}}}u\left(t,x\right)\right|^2 = \frac{\mu\sigma^2(B-A)}{\theta^2}, 
\end{equation}
with the detailed proof presented in Appendix~\ref{proof:eq:AsympDu}. Then, \eqref{eq:consistTheta2} and \eqref{eq:consistSigma2} follow at once. 

Asymptotic normality of $\widetilde{\theta}_{N,M}^2$ and  $\widetilde{\sigma}_{N,M}^2$ will be proved by a similar methods as in the proof of Theorem~\ref{th:est-ux-t}, and for brevity of writing some steps will be shortened. Let 
\begin{align*} 
\widetilde{U}(t,x_i)&:= \Delta_{\frac{B-A}{N^{\gamma}}}u(t,x_i)-\Delta_{\frac{B-A}{N^{\gamma}}}u(t,x_{i-1}){^2}\\
\mathbb{E}\widetilde{U}(t)&:=\mathbb{E}\left(\Delta_{\frac{B-A}{N^{\gamma}}}u(t,x_i)-\Delta_{\frac{B-A}{N^{\gamma}}}u(t,x_{i-1})\right)^2.
\end{align*}
Note that the right hand side of the last identity does not depend on $i$. Consequently, we put
$$
\widetilde{Q}_{N,M}:=\frac{1}{M}\sum_{j=1}^M\sum_{i=1}^{N} \left[\frac{\widetilde{U}(t_j,x_i)}{\mathbb{E} \widetilde{U}(t_j)}-1\right].
$$
Analogous to \eqref{eq:QN-multi}, we define the canonical Hilbert space $\widetilde{\mathcal{H}}$ associated to the Gaussian sequence 
$$
\left\{\displaystyle\frac{\Delta_{\frac{B-A}{N^{\gamma}}}u(t_j,x_k)-\Delta_{\frac{B-A}{N^{\gamma}}}u(t_j,x_{k-1})}{\left(\mathbb{E}\widetilde{U}(t_j)\right)^{1/2}},\ j=1,\dots, M,\ k=1,\dots,N\right\}_{M,N\in \mathbb{N}}
$$ 
and by the product formula for multiple stochastic integrals, we write
$$
\widetilde{Q}_{N,M}=\frac{1}{M}\sum_{j=1}^M
\sum_{i=1}^{N} \widetilde{I}^{t_j,t_j}_2\left(1_{[x_{i-1},x_i]}^{\otimes 2}\right),
$$
where $\widetilde{I}^{t_1,\dots,t_n}_n$ is the $n-$th multiple stochastic integral with a symmetric function $f$:
$$
\widetilde{I}^{t_1,\dots,t_n}_n(f):=\int_{\mathbb{R}^n} f(y_1,\dots,y_n)\dif \Delta_{\frac{B-A}{N^{\gamma}}}u(t_1,y_1)
\cdots \dif \Delta_{\frac{B-A}{N^{\gamma}}}u(t_n,y_n).
$$
Then, one can show that 
\begin{align}
\mathbb{E}\widetilde{Q}_{N,M}^2 =\frac{2N}{M}+ \widetilde{Q}_{1}+\widetilde{Q}_{2}+\widetilde{Q}_{ND}, \label{eq:QNM_decompose}
\end{align}
where
\begin{align*}
\widetilde{Q}_{1}&=\frac{4}{M^2}\sum_{\ell>k}
	\frac{1}{\mathbb{E} U(t_k)\mathbb{E} U(t_{\ell})} \sum_{i=1}^N \bigg[\mathbb{E}\left(\Delta_{\frac{B-A}{N^{\gamma}}}(t_k,x_i)-\Delta_{\frac{B-A}{N^{\gamma}}}(t_k,x_{i-1})\right)\\
	&\qquad \qquad \times \left(\Delta_{\frac{B-A}{N^{\gamma}}}(t_{\ell},x_i)-\Delta_{\frac{B-A}{N^{\gamma}}}(t_{\ell},x_{i-1})\right)\bigg]^2,\\
\widetilde{Q}_{2}
&=\frac{4}{M^2}\sum_{k=1}^M\left(\mathbb{E} \widetilde{U}(t_k)\mathbb{E} )\right)^{-2} 
\sum_{i>j} \bigg[\mathbb{E}\left(\Delta_{\frac{B-A}{N^{\gamma}}}u(t_k,x_i)-\Delta_{\frac{B-A}{N^{\gamma}}}u(t_k,x_{i-1})\right)\\
&\qquad \qquad  \times
\left(\Delta_{\frac{B-A}{N^{\gamma}}}u(t_{k},x_j)-\Delta_{\frac{B-A}{N^{\gamma}}}u(t_{k},x_{j-1})\right)\bigg]^2,\\
\widetilde{Q}_{ND}
&=\frac{8}{M^2}\sum_{\ell>k}\left(\mathbb{E} \widetilde{U}(t_k)\mathbb{E} \widetilde{U}(t_{\ell})\right)^{-1} 
\sum_{i>j} \bigg[\mathbb{E}\left(\Delta_{\frac{B-A}{N^{\gamma}}}u(t_k,x_i)-\Delta_{\frac{B-A}{N^{\gamma}}}u(t_k,x_{i-1})\right)\\
&\qquad \qquad  \times
\left(\Delta_{\frac{B-A}{N^{\gamma}}}u(t_{\ell},x_j)-\Delta_{\frac{B-A}{N^{\gamma}}}u(t_{\ell},x_{j-1})\right)\bigg]^2.
\end{align*}
Similar to $Q_{1}$, one can show that
	\begin{align*}
	\lim_{N\to \infty} \frac{1}{N}\widetilde{Q}_{1}=2-\frac{2}{M}.
	\end{align*}
Moreover, after a series of algebraic transformations, and by using Lemma~\ref{lemma:uhDuhproducs}, we obtain 
\begin{align}
\widetilde{Q}_{2} 
&= \frac{4}{M^2}\sum_{k=1}^M\left(\mathbb{E} \widetilde{U}(t_k)\right)^{-2} 
\sum_{i=1}^{N-1} (N-i)\widetilde{\Phi}_N^2(i,t_k,t_{k}), \label{eq:tildeQC2}\\
\widetilde{Q}_{ND} 
&= \frac{8}{M^2}\sum_{\ell>k}\left(\mathbb{E} \widetilde{U}(t_k)\mathbb{E} \widetilde{U}(t_{\ell})\right)^{-1} 
\sum_{i=1}^{N-1} (N-i)\widetilde{\Phi}_N^2(i,t_k,t_{\ell}), \label{eq:tildeQND}
\end{align}
where 
\begin{align*}
\widetilde{\Phi}_N(i,t_k,t_{\ell}) &:=\frac{\tau\theta N^{2\gamma}}{2(a+b)^2(B-A)^2} \bigg(2F(i,t_k,t_{\ell})-F(i+1,t_k,t_{\ell}) - F(i-1,t_k,t_{\ell})\bigg), \\
F\left(i,t_k,t_{\ell}\right)  &:=
\int_0^{t_k}\int_0^{t_{\ell}} (s_1+s_2)^{-1/2}\\
&\qquad \quad \quad \times \left(2e^{-\frac{i^2(B-A)^2}{4\theta  N^2(s_1+s_2)}}-e^{-\frac{\left(i+(a+b)N^{1-\gamma}\right)^2(B-A)^2}{4\theta N^2(s_1+s_2)}}-e^{-\frac{\left(i-(a+b)N^{1-\gamma}\right)^2(B-A)^2}{4\theta N^2(s_1+s_2)}}
\right)\dif s_1\dif s_2,
\end{align*}
for $i=1,\dots,N-1$ and $k,\ell=1,2,\dots,M$. We also define the normalized quadratic variation 
$$
\widetilde{\mathcal{Q}}_{N,M}:=\sqrt{\frac{1}{2N}}\widetilde{Q}_{N,M}.
$$

As proved in Appendix~\ref{proof:eqs:tildeQE-CLT1}, the following assertions hold true: 
\begin{align}
&\lim_{N\to \infty} \frac{1}{N}\widetilde{Q}_{2}= 0, \label{eq:LimTildeQC2} \\
&\lim_{N\to \infty} \frac{1}{N}\widetilde{Q}_{ND}= 0, \label{eq:LimTildeQND1} \\ 
&\lim_{N\to \infty}\frac{1}{N}\mathbb{E}\widetilde{Q}^2_{N,M} = 2, \label{eq:LimTildeQND2}
\end{align}
and for each $N,M\geq 1$, there exists a constant $C>0$ such that
\begin{equation}\label{eq:dTV2}
d_{TV}\left(\widetilde{\mathcal{Q}}_{N,M}, \mathcal{N}(0,1)\right)\leq \frac{C}{N}.
\end{equation}
Using these, we continue 
	\begin{align*}
	\widetilde{\mathcal{Q}}_{N,M}
	&=\frac{1}{M\sqrt{2}}\sum_{j=1}^M\left[\frac{\left(\sum_{i=1}^N 
		\left(\Delta_{\frac{B-A}{N^{\gamma}}}u(t_j,x_i)-\Delta_{\frac{B-A}{N^{\gamma}}}u(t_j,x_{i-1})\right)^2\right)-\frac{\mu\sigma^2(B-A)}{\theta^2}}{\sqrt{N}\mathbb{E}\left(\Delta_{\frac{B-A}{N^{\gamma}}}u(t_j,x_1)-\Delta_{\frac{B-A}{N^{\gamma}}}u(t_j,x_{0})\right)^2}\right]\\
	&\qquad +\frac{1}{M\sqrt{2}}\sum_{j=1}^M\left[\frac{\frac{\mu\sigma^2(B-A)}{\theta^2}-N\mathbb{E}\left(\Delta_{\frac{B-A}{N^{\gamma}}}u(t_j,x_i)-\Delta_{\frac{B-A}{N^{\gamma}}}u(t_j,x_{i-1})\right)^2}{\sqrt{N}\mathbb{E}\left(\Delta_{\frac{B-A}{N^{\gamma}}}u(t_j,x_1)-\Delta_{\frac{B-A}{N^{\gamma}}}u(t_j,x_0)\right)^2}\right], \\
	&\overset{a.s.}\simeq \frac{\sqrt{N}\theta^2}{M\sqrt{2}\mu\sigma^2(B-A)}\left[\sum_{j=1}^M\sum_{i=1}^N 
	\left(\Delta_{\frac{B-A}{N^{\gamma}}}u(t_j,x_i)-\Delta_{\frac{B-A}{N^{\gamma}}}u(t_j,x_{i-1})\right)^2-\frac{M\mu\sigma^2(B-A)}{\theta^2}\right]\\
	& \qquad + \frac{\sqrt{N}\theta^2}{M\sqrt{2}\mu\sigma^2(B-A)}\sum_{j=1}^M\left[
	\frac{\mu\sigma^2(B-A)}{\theta^2}-N\mathbb{E}\left(\Delta_{\frac{B-A}{N^{\gamma}}}u(t_j,x_i)-\Delta_{\frac{B-A}{N^{\gamma}}}u(t_j,x_{i-1})\right)^2 
	\right] \\
	&\overset{a.s.}\simeq 
	\sqrt{\frac{N}{2\sigma^4}}\left(\frac{1}{\mu}\widetilde{\sigma}^2_{N,M}-\sigma^2\right)\\
	&\qquad +
	\frac{\sqrt{N}\theta^2}{M\sqrt{2}\mu\sigma^2(B-A)}\sum_{j=1}^M\left[
	\frac{\mu\sigma^2(B-A)}{\theta^2}-N\mathbb{E}\left(\Delta_{\frac{B-A}{N^{\gamma}}}u(t_j,x_i)-\Delta_{\frac{B-A}{N^{\gamma}}}u(t_j,x_{i-1})\right)^2
	\right]\\
	&\overset{d}\longrightarrow \mathcal{N}(0,1), \quad \mbox{as}\ N\to \infty,
	\end{align*}
	which completes the proof of \eqref{eq:AsympNormSigma2}. Finally, \eqref{eq:AsympNormTheta2} follows by the delta method. 

\end{proof}

\subsection{The case of bounded domain}\label{sec:boundedStats}
The statistical analysis of \eqref{eq:1dPAM} on bounded domain and with Dirichlet boundary conditions was initiated in \cite{CialencoKimLototsky2019}, which also served as a basis for this study. Generally speaking, all results on estimation of $\theta$ and $\sigma$ stated above for $G=\bR$ remain true verbatim for Dirichlet boundary problems on $G=(0,\pi)$, despite substantial differences between  the form of the fundamental solutions of the heat equation $P$ and $P^D$. Formally, assume that the observed phenomena is governed by the evolution equation \eqref{eq:1dPAM}, with $G=(0,\pi)$ and endowed with zero boundary conditions. Also assume that the space grid points of $\Upsilon_{N,M}$ and $\widetilde{\Upsilon}_{N,M}$ belong to $G$. 

\begin{theorem}
Assume that $u_x$ is observed on $\Upsilon_{N,M}$. Then, the estimators $\wh{\theta}^2_{N,M}, \wh{\sigma}_{N,M}^2$, defined by \eqref{eq:t-estimator-theta}, \eqref{eq:t-estimator-sigma} satisfy \eqref{eq:tildeThetaConsist}-\eqref{eq:tildeSigmaAnNor}. Similarly, if $u$ is observed at grid-points $\wt\Upsilon_{N,M}$, then the estimators $\widetilde{\theta}^2_{N,M}, \widetilde{\sigma}_{N,M}^2$, defined by \eqref{eq:Dis-estimators-theta}, \eqref{eq:Dis-estimators-sigma} satisfy \eqref{eq:consistTheta2}-\eqref{eq:AsympNormSigma2}.
\end{theorem}
One of the advantages of the methods of proofs developed in this paper is that the main ideas, up to some technical computations, work for both the unbounded and bounded domain. 
For the sake of brevity we skip the (majority of) repetitive proofs of the auxiliary results, and focus on the major steps of the second part. 

First, we note that in view of \eqref{eq:mild-sol}  the mild solution $u$ in this case can be written as
\begin{equation}\label{eq:SolBounded}
	u(t,x) = \sqrt{\frac{2}{\pi}}\frac{\sigma}{\theta}\sum_{k\geq 1} \frac{1-e^{-k^2\theta t}}{k^2}\sin(kx)\xi_k,
\end{equation}
where 
$\set{\xi_k}_{k\in\bN}$ is a sequence of independent and identically distributed standard normal random variables. Without loss of generality, assume that $B-A=1$. We will show that $u$ satisfies \eqref{eq:AsympDu}. 
As before, we let $a,b\in[0,1], a+b>0$ in \eqref{eq:derApprox}. 
Using the representation \eqref{eq:SolBounded} and expanding $\Delta_{\frac{1}{N^{\gamma}}} u(t,x)$, after some elementary, albeit tedious, computations, we write 
\begin{align*}
\mathbb{E}\left(\Delta_{\frac{1}{N^{\gamma}}}u\left(t,x+\frac{1}{N}\right)-\Delta_{\frac{1}{N^{\gamma}}}u(t,x)\right)^2=:\widehat{A}_1+\widehat{A}_2,
\end{align*}
where
\begin{align*}
\widehat{A}_1&:=2\widehat{B}_1(0)-2\widehat{B}_1\left(\frac{1}{N}\right)+\widehat{B}_1\left(\frac{1}{N}+\frac{(a+b)}{N^{\gamma}}\right)+\widehat{B}_1\left(\frac{1}{N}-\frac{(a+b)}{N^{\gamma}}\right),\\
\widehat{A}_2&:=-\widehat{B}_2\left(\frac{2}{N}+\frac{2a}{N^{\gamma}}\right)+2\widehat{B}_2\left(\frac{1}{N}+\frac{2a}{N^{\gamma}}\right)+2\widehat{B}_2\left(\frac{2}{N}+\frac{(a-b)}{N^{\gamma}}\right)-4\widehat{B}_2\left(\frac{1}{N}+\frac{(a-b)}{N^{\gamma}}\right)\\
&\quad\ -\widehat{B}_2\left(\frac{2}{N}-\frac{2b}{N^{\gamma}}\right)+2\widehat{B}_2\left(\frac{1}{N}-\frac{2b}{N^{\gamma}}\right)-\widehat{B}_2\left(\frac{2a}{N^{\gamma}}\right)+2\widehat{B}_2\left(\frac{a-b}{N^{\gamma}}\right)-\widehat{B}_2\left(-\frac{b}{N^{\gamma}}\right),
\end{align*}
and
\begin{align*}
\widehat{B}_1(z)=\widehat{B}_1(N,t,z)& :=\frac{2\sigma^2 N^{2\gamma}}{\pi (a+b)^2\theta^2}\sum_{k\geq 1} \frac{\left(1-e^{-k^2\theta t}\right)^2}{k^4}\cos(kz), \\
\widehat{B}_2(z)=\widehat{B}_2(N,t,x,z)& := \frac{\sigma^2 N^{2\gamma}}{\pi (a+b)^2\theta^2}\sum_{k\geq 1} \frac{\left(1-e^{-k^2\theta t}\right)^2}{k^4} \cos\left(k\left(2x+z\right)\right).
\end{align*}
Let 
\begin{align*}
f(z)&:=\sum_{k\geq 1}  \frac{\left(1-e^{-k^2\theta t}\right)^2\cos(kz)}{k^2} \\
&=\sum_{k\geq 1}\frac{\cos(kz)}{k^2}-2\sum_{k\geq 1} \frac{e^{-k^2\theta t}\cos(kx)}{k^2}
+\sum_{k\geq 1} \frac{e^{-2k^2\theta t}\cos(kz)}{k^2},\ z\in(0,\pi).
\end{align*}
For each $t>0$, $x\in \mathbb{R}$ and $N\in \mathbb{N}$, the Taylor series of $\widehat{B}_1(z)$ at $z=0$ of order three gives that $\widehat{A}_1\simeq \displaystyle\frac{2\mu \sigma^2 f'(0)}{\pi \theta^2 N}$, as $N\to \infty$,
and by a similar argument $\widehat{A}_2 \sim\displaystyle\frac{1}{N^2}$, as $N\to \infty$. Since $f'(0)=\pi/2$, we achieve
	\begin{align}\label{eq:AsympDu-bdd}
	&\lim_{N\to \infty}N\mathbb{E}\left|\Delta_{\frac{1}{N^{\gamma}}}u\left(t,x+\frac{1}{N}\right)-\Delta_{\frac{1}{N^{\gamma}}}u(t,x)\right|^2=\frac{\mu\sigma^2}{\theta^2}.
	\end{align}
Thus, same as in Theorem~\ref{th:est-Du}, weak consistency of $\widetilde{\theta}^2_{N,M}$ and $\widetilde{\sigma}_{N,M}^2$ follow immediately. 

Next, 	using that $\mathbb{E}\xi_k\xi_{\ell}=1_{\{k=\ell\}}$ for any $k,\ell\geq 1$, combined with trigonometric identity  $\sin\alpha\sin\beta=(\cos(\alpha-\beta)-\cos(\alpha+\beta))/2$,
we deduce the counterpart of \eqref{eq:A1} that for $t,t'>0$ and $x,y\in (0,\pi)$, 
\begin{align*}
\mathbb{E}\Delta_hu(t,x)\Delta_hu(t',y)
& =\frac{\sigma^2}{\pi(a+b)^2h^2\theta^2}\bigg[2\widehat{F}(t,t',x-y)-\widehat{F}(t,t',x-y+(a+b)h)\\ 
& \qquad-\widehat{F}(t,t',x-y-(a+b)h) + 2\widehat{F}(t,t',x+y+(a-b)h)\\ 
& \qquad -\widehat{F}(t,t',x+y+2ah)-\widehat{F}(t,t',x+y-2bh)\bigg],
\end{align*}
where 
$$
\widehat{F}(t,t',x):=\sum_{k\geq 1} \frac{(1-e^{-k^2\theta t})(1-e^{-k^2\theta t'})}{k^4}\cos(kx).
$$
Same as in the proof of Theorem~\ref{th:est-Du}, we define $\widetilde{Q}_{N,M}$, and derive the decomposition \eqref{eq:QNM_decompose}.  
Let us focus on the nondiagonal term $\widetilde{Q}_{ND}$, which in this case can be represented as 
\begin{align*}
\widetilde{Q}_{ND}= \frac{8}{M^2}\sum_{\ell>k}^M\left(\mathbb{E} \widetilde{U}(t_k)\mathbb{E} \widehat{U}(t_{\ell})\right)^{-1} 
\sum_{i>j}^N\widehat{\Phi}_N^2(i,j,t_k,t_{\ell}),
\end{align*}
where we recall that $\mathbb{E}\widetilde{U}(t):=\mathbb{E}\left(\Delta_{\frac{1}{N^{\gamma}}}u(t,x_i)-\Delta_{\frac{1}{N^{\gamma}}}u(t,x_{i-1})\right)^2$
and where 
\begin{align*}
&\widehat{\Phi}_N(i,j,t_k,t_{\ell}):=\mathbb{E}\left(\Delta_{\frac{1}{N^{\gamma}}}u(t_k,x_i)-\Delta_{\frac{1}{N^{\gamma}}}u(t_k,x_{i-1})\right)\left(\Delta_{\frac{1}{N^{\gamma}}}u(t_{\ell},x_j)-\Delta_{\frac{1}{N^{\gamma}}}u(t_{\ell},x_{j-1})\right)\\
&=\frac{\sigma^2N^{2\gamma}}{\pi(a+b)^2\theta^2}\Bigg[4\widehat{F}\left(t_k,t_{\ell},\frac{i-j}{N}\right)-2\widehat{F}\left(t_k,t_{\ell},\frac{i-j}{N}+\frac{(a+b)}{N^{\gamma}}\right)-2\widehat{F}\left(t_k,t_{\ell},\frac{i-j}{N}-\frac{(a+b)}{N^{\gamma}}\right)\\
&-2\widehat{F}\left(t_k,t_{\ell},\frac{i-j+1}{N}\right)+\widehat{F}\left(t_k,t_{\ell},\frac{i-j+1}{N}+\frac{(a+b)}{N^{\gamma}}\right)+\widehat{F}\left(t_k,t_{\ell},\frac{i-j+1}{N}-\frac{(a+b)}{N^{\gamma}}\right)\\
&-2\widehat{F}\left(t_k,t_{\ell},\frac{i-j-1}{N}\right)+\widehat{F}\left(t_k,t_{\ell},\frac{i-j-1}{N}+\frac{(a+b)}{N^{\gamma}}\right)+\widehat{F}\left(t_k,t_{\ell},\frac{i-j-1}{N}-\frac{(a+b)}{N^{\gamma}}\right)\\
&+2\widehat{F}\left(t_k,t_{\ell},\frac{i+j}{N}+\frac{(a-b)}{N^{\gamma}}\right)-\widehat{F}\left(,t_k,t_{\ell},\frac{i+j}{N}+\frac{2a}{N^{\gamma}}\right)-\widehat{F}\left(t_k,t_{\ell},\frac{i+j}{N}-\frac{2b}{N^{\gamma}}\right)\\
&-4\widehat{F}\left(t_k,t_{\ell},\frac{i+j-1}{N}+\frac{(a-b)}{N^{\gamma}}\right)+2\widehat{F}\left(t_k,t_{\ell},\frac{i+j-1}{N}+\frac{2a}{N^{\gamma}}\right)+2\widehat{F}\left(t_k,t_{\ell},\frac{i+j-1}{N}-\frac{2b}{N^{\gamma}}\right)\\
&+2\widehat{F}\left(t_k,t_{\ell},\frac{i+j-2}{N}+\frac{(a-b)}{N^{\gamma}}\right)-\widehat{F}\left(t_k,t_{\ell},\frac{i+j-2}{N}+\frac{2a}{N^{\gamma}}\right)-\widehat{F}\left(t_k,t_{\ell},\frac{i+j-2}{N}-\frac{2b}{N^{\gamma}}\right)\Bigg].
\end{align*}
Using similar arguments in the proof of \eqref{eq:AsympDu-bdd}, we treat each term separately, and one can prove that
$\left|\widehat{\Phi}_N(i,j,t_k,t_{\ell})\right|\leq C/N^2$, 
for some constant $C>0$, independent of $i,j,t_k,t_{\ell}$. The terms $\widetilde{Q}_1$ and $\widetilde{Q}_2$ are treated similarly, and consequently \eqref{eq:LimTildeQC2}-\eqref{eq:dTV2} for this case are proved. The rest of the proof coincides with that from Theorem~\ref{th:est-Du}.


\section{Numerical experiments}\label{sec:numerical}
We consider the  equation \eqref{eq:1dPAM} on the bounded domain $G=(0,\pi)$ with zero boundary conditions, and with the following set of parameters: $\theta= 0.1, \ \sigma=0.1, \ T=1$. We will focus on estimating $\sigma$. 
Recall that in this case the solution is given by \eqref{eq:SolBounded}, which we use for our numerical approximations of the solution at any space-time point. We take $30,000$ Fourier coefficients or terms in the series appearing in \eqref{eq:SolBounded}, and the space sampling domain being the uniform grid of $G$. We use two finite difference approximation schemes in \eqref{eq:Dis-estimators-sigma}: (a) forward differences, by taking $a=1,b=0$, in which case the estimator is $\check\sigma_{N,M}^2$ given by \eqref{eq:Dis-estimators-sigma-forward} and the multiplicative bias $1/\mu=3/2$; (b) central differences, by taking $a=b=1$, now with $1/\mu = 12/5$ and the estimator $\widetilde\sigma^2_{N,M}$ is given by  \eqref{eq:Dis-estimators-sigma}.

In Figure~\ref{fig:Fig1} we display three sample paths of the $\sqrt{3/2}\check{\sigma}_{N,M}$ (left panel) and $\sqrt{12/5}\widetilde{\sigma}_{N,M}$ (right panel).   
On the horizontal axis $N$ is equal to the number of points in the discretization of $(0,\pi)$.  
Crossed marked lines correspond to the estimators that use the values of $u$ only at one time instance $t=0.2$ (i.e. $M=1$), while the stared lines depict the estimators that use data from $M=150$ time points, uniformly meshed between 0.2 and 1. The obtained results show that indeed $\sqrt{12/5}\widetilde{\sigma}_{N,M}$ and   $\sqrt{3/2}\check{\sigma}_{N,M}$ are consistent estimators of $\sigma$. This also shows that increasing $M$ does not change the overall behavior of the estimators. Similar results were obtained for $\widehat\sigma_{N,M}$ where the derivative $u_x$ is used. 

To validate the asymptotic normality of these estimators, we simulated $1,000$ paths of the solution $u(t,x)$ at $t=0.2$ and on a spatial grid with $N=1,000$ points. We compute $\check\sigma_{N,1}$ for each path. The normalized density of $\check\sigma_{N,1}$ superposed on the theoretical limiting density is presented in Figure~\ref{fig:Fig2} (left panel). The corresponding Q-Q plot is displayed in the right panel of the same figure. Overall, the obtained results validate the asymptotic normality of the $\check\sigma_{N,M}$. Similar results were obtained for other estimators. 

The numerical computations are performed using programming language Python. The source code is available at \url{https://github.com/cialenco/Stats4SPDEs-SHE}.

\begin{figure}[!ht]
	\centering
	\begin{subfigure}[t]{0.49\textwidth}
		\centering
		\includegraphics[width=\linewidth]{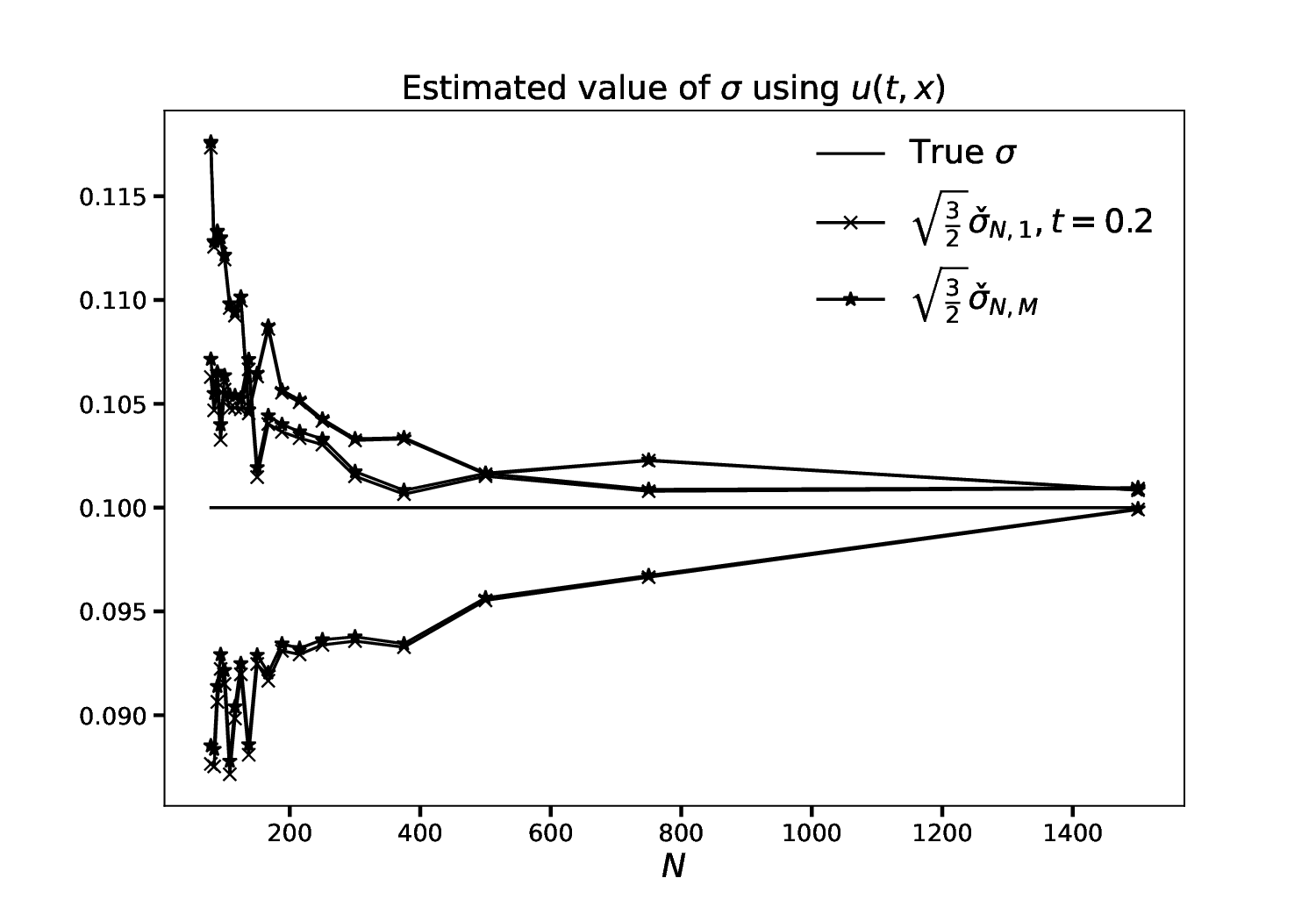}
	\end{subfigure}
	\hfill
	\begin{subfigure}[t]{0.49\textwidth}
		\centering
		\includegraphics[width=\linewidth]{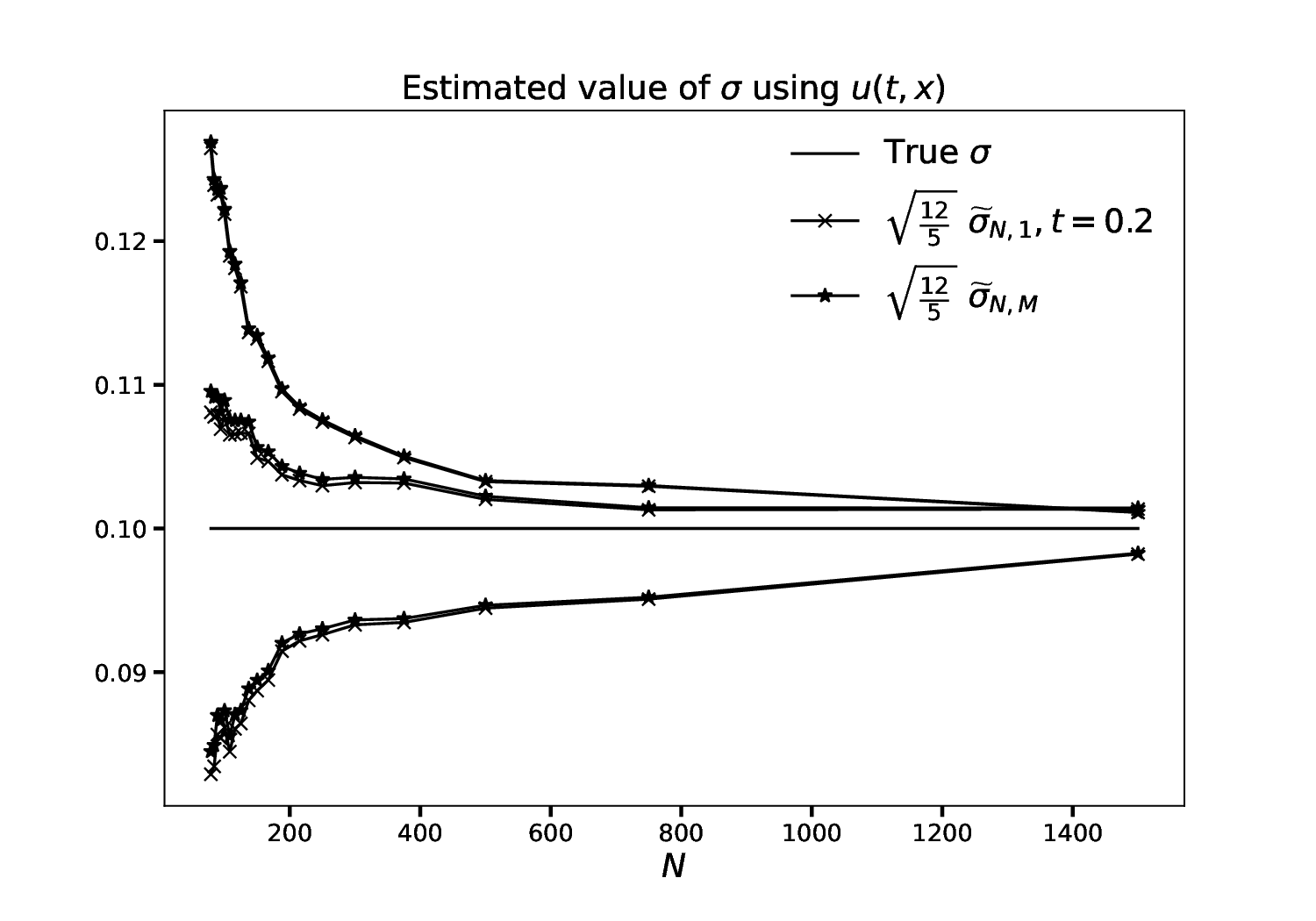}
	\end{subfigure}
	\caption{Three sample paths of $\sqrt{3/2}\check\sigma_{N,M}$ (left panel) and $\sqrt{12/5}\widetilde\sigma_{N,M}$ (right panel). Solid vertical line corresponds to the true parameter $\sigma=0.1$. }
	\label{fig:Fig1}
\end{figure}

\begin{figure}[!ht]
	\centering
	\begin{subfigure}[t]{0.49\textwidth}
		\centering
		\includegraphics[width=\linewidth]{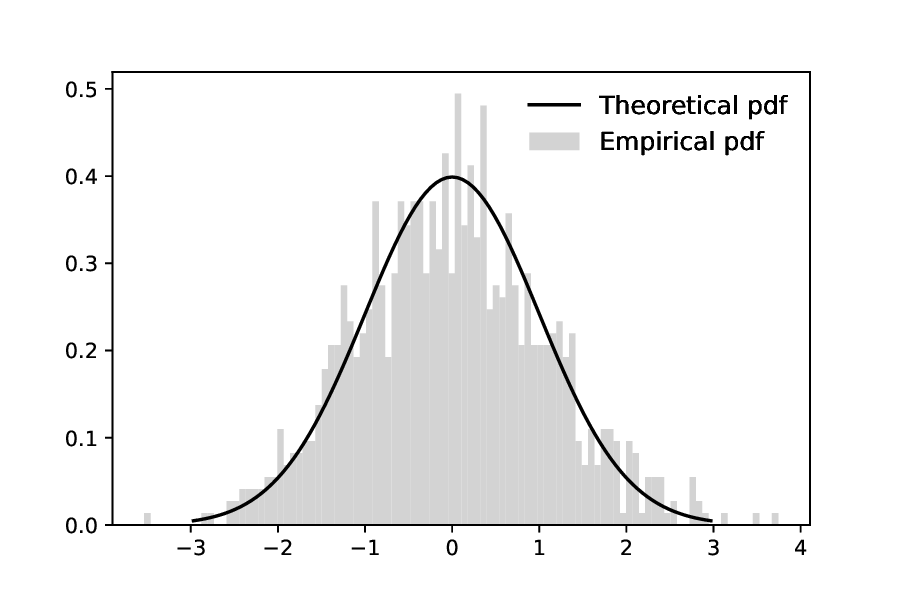}
	\end{subfigure}
	\hfill
	\begin{subfigure}[t]{0.49\textwidth}
		\centering
		\includegraphics[width=\linewidth]{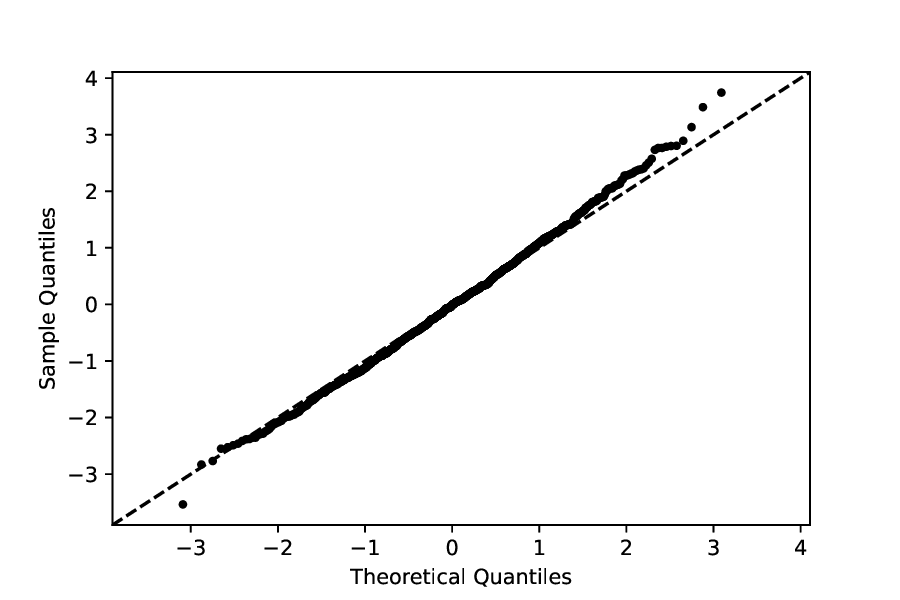}
	\end{subfigure}
	\caption{Left panel: empirical density function of the normalized estimator  $\sqrt{2N}{\sigma}(\frac{1}{\mu}\check\sigma_{N,1}-\sigma)$ (light gray) and the density of a standard Gaussian random variable (black curve). The corresponding Q-Q plot is displayed in the right panel.}
	\label{fig:Fig2}
\end{figure}

\section{Concluding remarks}\label{sec:future}
In this work we further study the stochastic heat equation driven by space-only noise that was initiated in \cite{CialencoKimLototsky2019}. Due to the special structure of the noise, this model yields several interesting statistical features, usually not seen in SPDEs driven by space-time noise. We showed that the ratio $\theta/\sigma$ can be estimated efficiently by measuring the solution in the physical space just at one time instant. However, the average of such estimators computed at several time points does not increase the quality of the estimation (the asymptotic variance does not decrease as $M$ gets larger), in contrast, for example, to the results from \cite{CialencoHuang2017,BibingerTrabs2017,BibingerTrabs2019}. On the other hand, the singular nature of the statistical problem when sampling at different time instances was noticed in \cite{CialencoKimLototsky2019}. In particular it was proved that for the case of the bounded domain $G=(0,\pi)$, the parameter $\theta$ can be determined exactly by knowing one Fourier mode of the solution at two time points, without assuming that $\sigma$ is known. Thus, for a fixed $k\in\bN$, and $0<t_1<t_2$, we consider the following  estimator
$$
\widehat\theta_N^F = \frac{\sum_{i=1}^{N} u(t_2,x_i)\sin(k^2 x_i)}{\sum_{i=1}^{N} u(t_1,x_i)\sin(k^2 x_i)},
$$
where $x_i$'s form a uniform partition of $(0,\pi)$. Using \cite[Theorem~3.2]{CialencoKimLototsky2019}, it is easy to show that $\widehat{\theta}^F_N$ is a consistent estimator of $F_{t_2,t_1}(k^2\theta)$, as $N\to\infty$, and where $F_{a,b}(x)=(1-e^{-ax})/(1-e^{-bx})$. The rate of convergence of this estimator remains an open problem.
Importantly from statistical point of view, in this case the joint estimation of $\theta$ and $\sigma$ in the discrete sampling scheme can be now disentangled, and first we estimate $\theta$  fairly simple, and then estimate $\sigma$ by one of the methods presented herein.  Noticeable, it is not clear at this stage how this argument can be transferred to the case of the full space $G=\bR$.

An important aspect of the obtained results is providing a pathway to statistical analysis of stochastic processes with smooth paths.  We showed that reducing smoothness by naive manipulations of derivatives may lead to nontrivial biases. In the recent paper \cite{CKP21}, the core idea developed in this manuscript was successfully applied, via some inductive steps, to semi-linear SPDEs driven by an additive noise, colored in space (of any degree).

An important extension of the current work is to consider nonlinear counterparts of \eqref{eq:1dPAM}, which will contribute to the limited literature (cf. \cite{IgorNathanAditiveNS2010,PasemannStannat2019,CKP21,ACP2020} and references therein) on statistical inference for nonlinear SPDEs.  The general line of attack is to use the same estimators from the linear equation, and prove that the solution of the nonlilnear equation can be represented as a smooth perturbation if its linear counterpart, hence preserving quadratic variation type arguments in derivation of the estimators.  

For simplicity, we assumed zero initial data, $u(0,x)=0, \ x\in G$, but morally, due to the same smooth perturbation arguments, the results remain true for any smooth initial data. This was also confirmed through numerical experiments.

\section*{Acknowledgments}
The authors are grateful to the editors and the anonymous referees for their helpful comments, suggestions, and insightful questions which helped to improve significantly the paper.

\appendix

\section{Appendix} \label{append:A}

In this section we present the detailed proofs for most of the results in this paper, as well as some needed technical results. For reader's convenience, at the end of this section, we give some known results from Malliavin calculus. We start with a result from Section~\ref{sec:regularity}.

\subsection{Proof of Lemma~\ref{lemma:uhDuhproducs}}\label{proof:lemma:uhDuhproducs}

Recall that for $t>0$, $x\in \mathbb{R}$ and $h>0$
$$
u(t,x)=\sigma\int_0^t\int_{\mathbb{R}} P(s,x-y)\dif W(y)\dif s, \quad \mbox{and}\quad \Delta_hu(t,x)=\sigma \int_0^t\int_{\mathbb{R}} \Delta_hP(s,x-y)\dif W(y)\dif s.	
$$
Then, \eqref{eq:A1} follows immediately by It\^o's isometry and the semigroup property of $P$. 
As far as \eqref{eq:A3}, we compute
\begin{align*}
\mathbb{E}v\left(t,x\right)v\left(t',y\right)
	&=\sigma^2
	\mathbb{E}\left(\int_0^t\int_{\mathbb{R}} P_x(s_1,x-z)\dif W(z)\dif s_1\right)
	\left(\int_0^{t'}\int_{\mathbb{R}} P_x(s_2,y-z)\dif W(z)\dif s_2\right)\\
	&=\frac{\sigma^2}{4\theta^2}\int_0^{t'}\int_0^t\int_{\mathbb{R}}\left(
	\frac{(x-z)(y-z)}{s_1s_2}P(s_1,x-z)P(s_2,y-z)
	\right)\dif z\dif s_1\dif s_2\\
	&=\frac{\sigma^2}{16\pi\theta^{3}}\int_0^{t'}\int_0^t\frac{1}{s_1^{3/2}s_2^{3/2}}
	\int_{\mathbb{R}} (x-z)(y-z)e^{-\frac{(x-z)^2}{4\theta s_1}-\frac{(y-z)^2}{4\theta s_2}}\dif z\dif s_1\dif s_2 \\
	&=\frac{\sigma^2}{16\pi\theta^{3}}\int_0^{t'}\int_0^t\frac{e^{-\frac{(x-y)^2}{4\theta(s_1+s_2)}}}{s_1^{3/2}s_2^{3/2}}
	\int_{\mathbb{R}} \left(z^2-(x-y)z\right)e^{-\frac{s_1+s_2}{4\theta s_1s_2}
		\left(z-\frac{(x-y)s_1}{s_1+s_2}\right)}\dif z\dif s_1\dif s_2.
	\end{align*}
	Using the mean and the variance of Gaussian random variables, we get
	\begin{align*}
	\mathbb{E}v\left(t,x\right)v\left(t',y\right)
	&=\frac{\sigma^2}{4\sqrt{\pi}\theta^{3/2}}\int_0^{t'}\int_0^t\frac{e^{-\frac{(x-y)^2}{4\theta(s_1+s_2)}}}{(s_1+s_2)^{3/2}}
	\left(1-\frac{(x-y)^2}{2\theta(s_1+s_2)}\right)\dif s_1\dif s_2.
	\end{align*}

Identity \eqref{eq:A2} is obtained by analogous evaluations: 
\begin{align*}
\mathbb{E}& \Delta_hu\left(t,x\right)v\left(t',x\right)
	=\sigma^2
	\mathbb{E}\left(\int_0^t\int_{\mathbb{R}} \Delta_hP(s_1,x-y)\dif W(y)\dif s_1\right)\left(\int_0^{t'}\int_{\mathbb{R}} P_x(s_2,x-y)\dif W(y)\dif s_2\right)\\
	&=-\frac{\sigma^2}{2\theta (a+b) h}\int_0^{t'}\int_0^t\int_{\mathbb{R}}
	\frac{x-y}{s_2}\left(P(s_1,x-y+ah)-P(s_1,x-y-bh)\right)P(s_2,x-y)\dif y\dif s_1\dif s_2\\
	&=-\frac{\sigma^2}{8\pi\theta^2(a+b) h}\int_0^{t'}\int_0^t\frac{1}{s_1^{1/2}s_2^{3/2}}
	\int_{\mathbb{R}} (x-y)\left[e^{-\frac{(x-y+ah)^2}{4\theta s_1}-\frac{(x-y)^2}{4\theta s_2}}
	-e^{-\frac{(x-y-bh)^2}{4\theta s_1}-\frac{(x-y)^2}{4\theta s_2}}\right]\dif y\dif s_1\dif s_2\\
	&=-\frac{\sigma^2}{8\pi \theta^2(a+b) h}\int_0^{t'}\int_0^t \frac{e^{-\frac{a^2h^2}{4\theta (s_1+s_2)}}}{s_1^{1/2}s_2^{3/2}}
	\int_{\mathbb{R}} z\cdot e^{-\frac{s_1+s_2}{4\theta s_1s_2}\left(z+\frac{as_2h}{s_1+s_2}\right)^2} \dif z\dif s_1\dif s_2\\
	&\quad+\frac{\sigma^2 }{8\pi \theta^2(a+b) h}\int_0^{t'}\int_0^t \frac{e^{-\frac{b^2h^2}{4\theta (s_1+s_2)}}}{s_1^{1/2}s_2^{3/2}}
	\int_{\mathbb{R}} z\cdot e^{-\frac{s_1+s_2}{4\theta s_1s_2}\left(z-\frac{bs_2h}{s_1+s_2}\right)^2} \dif z\dif s_1\dif s_2\\
	&=\frac{a\sigma^2}{4\sqrt{\pi}\theta^{3/2}(a+b)}\int_0^{t'}\int_0^t\frac{e^{-\frac{a^2h^2}{4\theta(s_1+s_2)}}}{(s_1+s_2)^{3/2}} \dif s_1\dif s_2
	+\frac{b\sigma^2}{4\sqrt{\pi}\theta^{3/2}(a+b)}\int_0^{t'}\int_0^t\frac{e^{-\frac{b^2h^2}{4\theta(s_1+s_2)}}}{(s_1+s_2)^{3/2}} \dif s_1\dif s_2.
	\end{align*}
The proof is complete. 

\subsection{Proofs of results from Section~\ref{sec:stats}}\label{append:ProofsStats}

Without loss of generality, in all technical proofs in this section, we will set $A=0$ and $B=1$. That is, $0=x_0<x_1<\cdots <x_N=1$ and $x_i=i/N$. 

\subsubsection{Proof of Proposition~\ref{prop:uxasymp}}\label{proof:prop:uxasymp}
By Lemma~\ref{lemma:uhDuhproducs}, for $t>0$, $x\in \mathbb{R}$ and $N\geq 1$,
\begin{align*}
B_1&:=\mathbb{E}\left|u_x\left(t,x+\frac{1}{N}\right)-u_x(t,x)\right|^2\\
&=\frac{\tau}{2}\int_0^t\int_0^t 
\frac{1-e^{-\frac{1}{4\theta N^2(s_1+s_2)}}}{(s_1+s_2)^{3/2}}\dif s_1\dif s_2
+\frac{\tau}{4\theta N^2}
\int_0^{t}\int_0^{t} \frac{e^{-\frac{1}{4\theta N^2(s_1+s_2)}}}{(s_1+s_2)^{5/2}}
\dif s_1\dif s_2.
\end{align*}
Consequently, by Lemma~\ref{lemma:intbyparts-1}, we get
\begin{align}
B_1 & =-2\sqrt{2t}\tau\left(1-e^{-\frac{1}{8\theta N^2t}}\right)
+4\sqrt{t}\tau \left(1-e^{-\frac{1}{4\theta N^2t}}\right) \nonumber \\
& \quad 
 +\frac{\tau}{\sqrt{\theta}N}\int_0^{4tN^{2}\theta} \frac{e^{-\frac{1}{s}}}{s^{3/2}}\dif s
-\frac{\tau}{\sqrt{\theta}N}\int_{4tN^{2}\theta}^{8tN^{2}\theta} 
\frac{e^{-\frac{1}{s}}}{s^{3/2}}\dif s.\label{2ndmoment}
\end{align}
Since $1-e^{-\frac{\beta}{N^{2}}}\simeq \beta/N^{2}$, for $\beta>0$, as $N\to \infty$,
we continue 
\begin{equation*}
NB_1-\frac{\tau}{\sqrt{\theta}}\int_0^{\infty}
\frac{e^{-\frac{1}{s}}}{s^{3/2}}\dif s\sim N^{-1},\quad \mbox{as}\ N\to \infty.
\end{equation*}
Note that 
\begin{equation*}
\int_0^{\infty}\frac{e^{-\frac{1}{s}}}{s^{3/2}}\dif s=\int_0^{\infty} s^{-1/2}e^{-s}\dif s=\Gamma(1/2)=\sqrt{\pi},
\end{equation*}
which implies that 
\begin{equation*}
\lim_{N\to \infty}N^{\alpha}\left[N B_1 -\frac{\sigma^2}{\theta^{2}}\right]= 0 
\end{equation*}
for any $\alpha<1$.
This completes the proof. 

\subsubsection{Proof of identities \eqref{eq:QC2-mainText} and \eqref{eq:QND-mainText}}\label{proof:eq3-9-3-10}
As usual we denote by ${\rm{erf}}$ the error function  
$$
{\rm{erf}}(x)=\frac{2}{\sqrt{\pi}}\int_0^x e^{-y^2}\dif y.
$$

\noindent
\textit{Step I.} First we will prove that for a uniform partition  $\set{x_i \mid i=0,...,N}\subset[0,1]$, and $N\in \mathbb{N}$, for any $i>j$ and $k,\ell=1,2,\dots,M$, we have
\begin{align}
I:=\mathbb{E}\big(u_x(t_k,x_i) & -u_x(t_k,x_{i-1})\big)
\left(u_x(t_{\ell},x_j)-u_x(t_{\ell},x_{j-1})\right)\nonumber\\
&=\tau \sqrt{t_{\ell}}\varphi_{\frac{1}{4\theta t_{\ell}}}
		\left(\frac{1}{N},\frac{i-j}{N}\right)+\tau \sqrt{t_{k}}\varphi_{\frac{1}{4\theta t_{k}}}
		\left(\frac{1}{N},\frac{i-j}{N}\right) \nonumber\\
&		\quad -\sqrt{t_{\ell}+t_k}\tau \varphi_{\frac{1}{4\theta (t_{\ell}+t_k)}}\left(\frac{1}{N},\frac{i-j}{N}\right)
+\frac{\tau \sqrt{\pi}}{2\sqrt{\theta}}\phi_{(4\theta t_{\ell},4\theta t_k)}\left(\frac{1}{N},\frac{i-j}{N}\right), \label{eq:Euxk}
\end{align}
where, for $c,c_1,c_2,x,z> 0$, 
\begin{align}
\varphi_{c}(z,x)&:= \left[2e^{-cx^2} -e^{-c(x+z)^2}-e^{-c(x-z)^2} \right] \label{eq:phic}\\
\phi_{(c_1,c_2)}(z,x) & := 
2x\left({\rm{erf}}\left(\frac{x}{\sqrt{c_1}}\right) + {\rm{erf}}\left(\frac{x}{\sqrt{c_2}}\right)
-{\rm{erf}}\left(\frac{x}{\sqrt{c_1+c_2}}\right)-1 	\right) \nonumber\\
&\quad -(x+z)  \left({\rm{erf}}\left(\frac{x+z}{\sqrt{c_1}}\right) +{\rm{erf}}\left(\frac{x+z}{\sqrt{c_2}}\right) - {\rm{erf}}\left(\frac{x+z}{\sqrt{c_1+c_2}}\right) -1 \right)\nonumber\\
&\quad -(x-z)\left({\rm{erf}}\left(\frac{x-z}{\sqrt{c_1}}\right) + {\rm{erf}}\left(\frac{x-z}{\sqrt{c_2}}\right) - {\rm{erf}}\left(\frac{x-z}{\sqrt{c_1+c_2}}\right)-1
		\right).\label{eq:phic1c2}
\end{align}
Indeed, since for $i>j$, we have
\begin{align*}
x_i-x_j=x_{i-1}-x_{j-1}=(i-j)/N,\\
x_i-x_{j-1}=(i-j+1)/N,\quad x_{i-1}-x_j=(i-j-1)/N.
\end{align*}
then, for $i>j$, we obtain that 
\begin{align*}
I&=\tau\int_0^{t_{\ell}}\int_0^{t_k} \Bigg[\frac{e^{-\frac{(i-j)^2}{4\theta N^2(s_1+s_2)}}}{2(s_1+s_2)^{3/2}} \left(1-\frac{(i-j)^2}{2\theta N^2(s_1+s_2)}\right) 	-\frac{e^{-\frac{(i-j+1)^2}{4\theta N^2(s_1+s_2)}}}{4(s_1+s_2)^{3/2}} \left(1-\frac{(i-j+1)^2}{2\theta N^2(s_1+s_2)}\right)\\ 
& \quad -\frac{e^{-\frac{(i-j-1)^2}{4\theta N^2(s_1+s_2)}}}{4(s_1+s_2)^{3/2}} \left(1-\frac{(i-j-1)^2}{2\theta N^2(s_1+s_2)}\right)\Bigg] \dif s_1\dif s_2\\
& =: \bar{A}_1+\bar{A}_2+\bar{A}_3.
\end{align*}
Next we note that for $c,t_{\ell},t_k,x\geq 0$, in view of Lemma~\ref{lemma:intbyparts-1}, we get 	\begin{align*}
\int_0^{t_{\ell}}\int_0^{t_k}&\left((s_1+s_2)^{-3/2}e^{-\frac{cx^2}{s_1+s_2}}-2cx^2(s_1+s_2)^{-5/2}e^{-\frac{cx^2}{s_1+s_2}}\right)\dif s_1\dif s_2 \\
&=4\sqrt{t_{\ell}}e^{-\frac{cx^2}{t_{\ell}}}+4\sqrt{t_{k}}e^{-\frac{cx^2}{t_{k}}}-4\sqrt{t_{\ell}+t_k}e^{-\frac{cx^2}{t_{\ell}+t_k}} \\
&+4\sqrt{c}x\left(\int_{\frac{t_k}{cx^2}}^{\frac{t_{\ell}+t_k}{cx^2}} s^{-3/2}e^{-\frac{1}{s}}\dif s-\int_0^{\frac{t_{\ell}}{cx^2}} s^{-3/2}e^{-\frac{1}{s}}\dif s\right)\\
&=4\sqrt{t_{\ell}}e^{-\frac{cx^2}{t_{\ell}}}+4\sqrt{t_{k}}e^{-\frac{cx^2}{t_{k}}}-4\sqrt{t_{\ell}+t_k}e^{-\frac{cx^2}{t_{\ell}+t_k}} \\
&+4\sqrt{c\pi}x\left(\mbox{erf}\left(\frac{\sqrt{c}x}{\sqrt{t_{\ell}}}\right)+
\mbox{erf}\left(\frac{\sqrt{c}x}{\sqrt{t_k}}\right)-\mbox{erf}\left(\frac{\sqrt{c}x}{\sqrt{t_{\ell}+t_k}}\right)-1\right).
\end{align*}
From here, by direct computations, we obtain 
\begin{align*}
\bar{A}_1 & =2\tau\sqrt{t_{\ell}}e^{-\frac{(i-j)^2}{4\theta N^2t_{\ell}}}+ 2\tau\sqrt{t_{k}}e^{-\frac{(i-j)^2}{4\theta N^2t_{k}}} - 2\tau\sqrt{t_{\ell}+t_k}e^{-\frac{(i-j)^2}{4\theta N^2(t_{\ell}+t_k)}}\\
& \quad +\frac{\sqrt{\pi}(i-j)\tau}{\sqrt{\theta}N} \left(\mbox{erf}\left(\frac{i-j}{2\sqrt{\theta t_{\ell}}N}\right)+\mbox{erf}\left(\frac{i-j}{2\sqrt{\theta  t_{k}}N}\right) - \mbox{erf}\left(\frac{i-j}{2\sqrt{\theta (t_{\ell}+t_k)}N}\right) -1\right),\\
\bar{A}_2 & =-\tau\sqrt{t_{\ell}}e^{-\frac{(i-j+1)^2}{4\theta N^2t_{\ell}}} -\tau\sqrt{t_{k}}e^{-\frac{(i-j+1)^2}{4\theta N^2t_{k}}}
+\sqrt{t_{\ell}+t_k}\tau e^{-\frac{(i-j+1)^2}{4\theta N^2(t_{\ell}+t_k)}}\\
&\quad -\frac{\sqrt{\pi}(i-j+1)\tau}{2\sqrt{\theta} N}
\left(\mbox{erf}\left(\frac{i-j+1}{2\sqrt{\theta t_{\ell}}N}\right)+\mbox{erf}\left(\frac{i-j+1}{2\sqrt{\theta t_{k}}N}\right)-\mbox{erf}\left(\frac{i-j+1}{2\sqrt{\theta (t_{\ell}+t_k)}N}\right)-1\right),\\
\bar{A}_3 & =-\tau\sqrt{t_{\ell}}e^{-\frac{(i-j-1)^2}{4\theta N^2t_{\ell}}} -\tau\sqrt{t_{k}}e^{-\frac{(i-j-1)^2}{4\theta N^2t_{k}}} + \sqrt{t_{\ell}+t_k}\tau e^{-\frac{(i-j-1)^2}{4\theta N^2(t_{\ell}+t_k)}}\\
&\quad -\frac{\sqrt{\pi}(i-j-1)\tau}{2\sqrt{\theta}N} \left(\mbox{erf}\left(\frac{i-j-1}{2\sqrt{\theta t_{\ell}}N}\right)+\mbox{erf}\left(\frac{i-j-1}{2\sqrt{\theta t_{k}}N}\right)-\mbox{erf}\left(\frac{i-j-1}{2\sqrt{\theta (t_{\ell}+t_k)}N}\right)-1\right). 
\end{align*}
After combining the above, one gets \eqref{eq:Euxk}. 

\smallskip
\noindent
\textit{Step 2.} 	
Let 
\begin{align}
\Phi_N(i,t_{\ell},t_k)&:=\tau \sqrt{t_{\ell}}\varphi_{\frac{1}{4\theta t_{\ell}}}
\left(\frac{1}{N},\frac{i}{N}\right)+\tau \sqrt{t_{k}}\varphi_{\frac{1}{4\theta t_{k}}}
\left(\frac{1}{N},\frac{i}{N}\right)
-\sqrt{t_{\ell}+t_k}\tau \varphi_{\frac{1}{4\theta (t_{\ell}+t_k)}}\left(\frac{1}{N},\frac{i}{N}\right)\nonumber\\
&+\frac{\tau \sqrt{\pi}}{2\sqrt{\theta}}\phi_{(4\theta t_{\ell},4\theta t_k)}\left(\frac{1}{N},\frac{i}{N}\right), \label{eq:PhiN}
\end{align}
where $\phi_c$ and $\phi_{(c_1,c_2)}$ are defined in \eqref{eq:phic} and \eqref{eq:phic1c2}. 
Then, substituting \eqref{eq:Euxk} directly into the definitions of $Q_{2}$ and $Q_{ND}$, the identities \eqref{eq:QC2-mainText} and \eqref{eq:QND-mainText} follow at once. The proof is complete. 
	

\subsubsection{Proof of Theorem~\ref{th:CLT}}\label{proof:th:CLT}
By \eqref{eq:QN-multi}, and taking into account that $\cQ_{N,M}=\sqrt{1/2N}Q_{N,M}$, we deduce that  Malliavin derivative $\mathcal{D}\mathcal{Q}_{N,M}$ is given by
$$
\mathcal{D}\mathcal{Q}_{N,M}=\frac{2}{M\sqrt{2N}}\sum_{j=1}^M
\sum_{i=1}^N I_1^{t_j}\left(1_{[x_{i-1},x_i]}\right)1_{[x_{i-1},x_i]}.
$$
Hence, by the product formula for multiple integrals, we deduce 
\begin{align*}
\left\|\mathcal{D}\mathcal{Q}_{N,M}\right\|_{\mathcal{H}}^2 
& = \frac{2}{M^2N}\sum_{k,\ell=1}^M\left(\mathbb{E} U(t_k)\mathbb{E} U(t_{\ell})\right)^{-1/2} 
\sum_{i,j=1}^N  I_1^{t_k}\left(1_{[x_{i-1},x_i]}\right) I_1^{t_{\ell}}\left(1_{[x_{j-1},x_j]}\right)\\
& \qquad \times\mathbb{E}\left(u_x(t_k,x_i)-u_x(t_k,x_{i-1})\right)\left(u_x(t_{\ell},x_j) -u_x(t_{\ell},x_{j-1})\right)\\
& = \frac{2}{M^2N}\sum_{k,\ell=1}^M\left(\mathbb{E} U(t_k)\mathbb{E} U(t_{\ell})\right)^{-1/2} 
\sum_{i,j=1}^N I_2^{t_k,t_{\ell}}\left(1_{[x_{i-1},x_i]}\otimes 1_{[x_{j-1},x_j]}\right)\\
&\qquad \times \mathbb{E}\left(u_x(t_k,x_i)-u_x(t_k,x_{i-1})\right)\left(u_x(t_{\ell},x_j) -u_x(t_{\ell},x_{j-1})\right)+\mathbb{E}\|\mathcal{D}\mathcal{Q}_{N,M}\|_{\mathcal{H}}^2.
\end{align*}
Let 
\begin{align*}
\mathcal{R}_{N,M} & :=\|\mathcal{D}\mathcal{Q}_{N,M}\|_{\mathcal{H}}^2-\mathbb{E}\|\mathcal{D}\mathcal{Q}_{N,M}\|_{\mathcal{H}}^2\\
& = \frac{2}{{{M^2N}}}\sum_{k,\ell=1}^M\left(\mathbb{E} U(t_k)\mathbb{E} U(t_{\ell})\right)^{{-1/2}}  
 \sum_{i,j=1}^N I_2^{t_k,t_{\ell}}\left(1_{[x_{i-1},x_i]}\otimes 1_{[x_{j-1},x_j]}\right)\\
& \qquad \times \mathbb{E}\left(u_x(t_k,x_i)-u_x(t_k,x_{i-1})\right) \left(u_x(t_{\ell},x_j)-u_x(t_{\ell},x_{j-1})\right).
\end{align*}
Then, in view of the above and by Theorem \ref{MSthm}, we have
\begin{equation}\label{eq:dTV3}
d_{TV}\left(\mathcal{Q}_{N,M}, \mathcal{N}(0,1)\right)\leq C_1\left(\mathbb{E}\mathcal{R}_{N,M}^2\right)^{1/2}, \end{equation}
for some $C_1>0$. Thus, we finally compute
\begin{align*}
\mathbb{E}\mathcal{R}^2_{N,M}= \frac{8}{{{M^4N^2}}}&\sum_{i',j',k',\ell'=1}^M\sum_{i,j,k,\ell=1}^N \left(\mathbb{E} U(t_{i'})\mathbb{E} U(t_{j'})\mathbb{E} U(t_{k'})\mathbb{E} U(t_{\ell'})\right)^{-1}\\
&\times \mathbb{E}\left(u_x(t_{i'},x_i)-u_x(t_{i'},x_{i-1})\right)\left(u_x(t_{k'},x_k)-u_x(t_{k'},x_{k-1})\right)\\
&\times \mathbb{E}\left(u_x(t_{j'},x_j)-u_x(t_{j'},x_{j-1})\right)\left(u_x(t_{\ell'},x_{\ell})-u_x(t_{\ell'},x_{\ell-1})\right)\\
&\times \mathbb{E}\left(u_x(t_{i'},x_i)-u_x(t_{i'},x_{i-1})\right)\left(u_x(t_{j'},x_j)-u_x(t_{j'},x_{j-1})\right)\\
&\times \mathbb{E}\left(u_x(t_{k'},x_k)-u_x(t_{k'},x_{k-1})\right)\left(u_x(t_{\ell'},x_{\ell})-u_x(t_{\ell'},x_{\ell-1})\right).
\end{align*}
Note that
\begin{align*}
\mathbb{E}\mathcal{R}^2_{N,M}&=\frac{8}{{{M^4N^2}}}
\sum_{i',j',k',\ell'=1}^M\sum_{i,j,k,\ell=1}^N
\left(\mathbb{E} U(t_{i'})\mathbb{E} U(t_{j'})\mathbb{E} U(t_{k'})\mathbb{E} U(t_{\ell'})\right)^{-1}\\
&\times \Phi_N\left(|i-k|,t_{i'},t_{k'}\right)\Phi_N\left(|j-l|,t_{j'},t_{\ell'}\right) 
\Phi_N\left(|i-j|,t_{i'},t_{j'}\right)\Phi_N\left(|k-l|,t_{k'},t_{j'}\right).
\end{align*}
Since, by \eqref{Phi-taylor-1}, for all $i=1,2,\dots,N$ and $k,\ell=1,\dots,M$, 
$$
\Phi_N(i,t_k,t_{\ell})\leq \frac{C_2}{\sqrt{t_1}N^{2}}
\left(\frac{(i+1)^2}{t_1N^{2}}+1\right), 
$$
we immediately conclude that 
$$
\mathbb{E}\mathcal{R}^2_N \leq \frac{C_3}{N^2}.
$$
This, combined with \eqref{eq:dTV3} concludes the proof. 

\subsubsection{Proof of \eqref{eq:AsympDu}}\label{proof:eq:AsympDu}

By Lemma~\ref{lemma:uhDuhproducs}, for $t>0$, $x\in \mathbb{R}$ and $N\geq 1$, we get
\begin{align}
B_2&:=\mathbb{E}\left|\Delta_{\frac{1}{N^{\gamma}}}u\left(t,x+\frac{1}{N}\right)-\Delta_{\frac{1}{N^{\gamma}}}u(t,x)\right|^2\nonumber\\
& = \frac{2\tau \theta N^{2\gamma}}{(a+b)^2}\int_0^t\int_0^t 
(s_1+s_2)^{-1/2}\left(1-e^{-\frac{(a+b)^2}{4\theta N^{2\gamma}(s_1+s_2)}}\right)
\dif s_1\dif s_2\nonumber\\
&\qquad +\frac{\tau \theta N^{2\gamma}}{(a+b)^2}\int_0^t\int_0^t 
(s_1+s_2)^{-1/2}\left(e^{-\frac{(a+b+N^{\gamma-1})^2}{4\theta N^{2\gamma}(s_1+s_2)}}-e^{-\frac{1}{4\theta N^2(s_1+s_2)}}\right)
\dif s_1\dif s_2\nonumber\\
&\qquad +\frac{\tau \theta N^{2\gamma}}{(a+b)^2}\int_0^t\int_0^t 
(s_1+s_2)^{-1/2}\left(e^{-\frac{(a+b-N^{\gamma-1})^2}{4\theta N^{2\gamma}(s_1+s_2)}}-e^{-\frac{1}{4\theta N^2(s_1+s_2)}}\right)
\dif s_1\dif s_2 \nonumber \\
&=:\widetilde{A}_1+\widetilde{A}_2+\widetilde{A}_3. \label{eq:EDhu2}
\end{align}
We deal with each term in \eqref{eq:EDhu2} by applying consequently Lemma \ref{lemma:intbyparts-1}, and write 
\begin{align}\label{eq:EDuh-1}
\widetilde{A}_1&=\frac{8\tau\theta N^{2\gamma}}{3(a+b)^2}(2t)^{3/2}\left(1-e^{-\frac{(a+b)^2}{8\theta N^{2\gamma}t}}\right)
-\frac{16\tau\theta N^{2\gamma}}{3(a+b)^2}t^{3/2}\left(1-e^{-\frac{(a+b)^2}{4\theta N^{2\gamma}t}}\right)
-\frac{8\tau}{3}(2t)^{1/2}e^{-\frac{(a+b)^2}{8\theta N^{2\gamma}t}}\nonumber\\
& \quad +\frac{16\tau}{3}t^{1/2}e^{-\frac{(a+b)^2}{4\theta N^{2\gamma}t}}
+\frac{4\tau(a+b)}{3\sqrt{\theta}N^{\gamma}}\left[\int_{\frac{4\theta t N^{2\gamma}}{(a+b)^2}}^{\frac{8\theta t N^{2\gamma}}{(a+b)^2}}
s^{-3/2}e^{-1/s}\dif s-\int_0^{\frac{4\theta t N^{2\gamma}}{(a+b)^2}}s^{-3/2}e^{-1/s}\dif s
\right]\nonumber\\
& \quad +\frac{\tau (a+b)}{\sqrt{\theta}N^{\gamma}}\int_0^{\frac{4\theta t N^{2\gamma}}{(a+b)^2}}\int_0^{\frac{4\theta t N^{2\gamma}}{(a+b)^2}} (s_1+s_2)^{-5/2}e^{-\frac{1}{s_1+s_2}}\dif s_1\dif s_2,
\end{align}
\begin{align}\label{eq:EDuh-2}
\widetilde{A}_2&=\frac{4\tau\theta N^{2\gamma}}{3(a+b)^2}(2t)^{3/2}\left(e^{-\frac{(1+(a+b)N^{1-\gamma})^2}{8\theta N^{2}t}}-e^{-\frac{1}{8\theta N^2t}}\right)
-\frac{8\tau\theta N^{2\gamma}}{3(a+b)^2}t^{3/2}\left(e^{-\frac{(1+(a+b)N^{1-\gamma})^2}{4\theta N^{2}t}}-e^{-\frac{1}{4\theta N^2t}}\right) \nonumber\\
&\quad +\frac{4\tau (1+(a+b)N^{1-\gamma})^2N^{2\gamma-2}}{3(a+b)^2}(2t)^{1/2}e^{-\frac{(1+(a+b)N^{1-\gamma})^2}{8\theta N^{2}t}}-\frac{4\tau N^{2\gamma-2}}{3(a+b)^2}(2t)^{1/2}e^{-\frac{1}{8\theta N^2t}} \nonumber\\
&\quad -\frac{8\tau (1+(a+b)N^{1-\gamma})^2N^{2\gamma-2}}{3(a+b)^2}t^{1/2}e^{-\frac{(1+(a+b)N^{1-\gamma})^2}{4\theta N^{2}t}}+\frac{8\tau N^{2\gamma-2}}{3(a+b)^2}t^{1/2}e^{-\frac{1}{4\theta N^2t}} \nonumber\\
& \quad - \frac{2\tau(1+(a+b)N^{1-\gamma})^3}{3(a+b)^2\sqrt{\theta}N^{3-2\gamma}}\left[\int_{\frac{4\theta t N^{2}}{(1+(a+b)N^{1-\gamma})^2}}^{\frac{8\theta t N^{2}}{(1+(a+b)N^{1-\gamma})^2}}
s^{-3/2}e^{-1/s}\dif s-\int_0^{\frac{4\theta t N^{2}}{(1+(a+b)N^{1-\gamma})^2}}s^{-3/2}e^{-1/s}\dif s
\right]\nonumber\\
&\quad +\frac{2\tau}{3(a+b)^2\sqrt{\theta} N^{3-2\gamma}}\left[\int_{4\theta t N^2}^{8\theta t N^2}
s^{-3/2}e^{-1/s}\dif s-\int_0^{4\theta t N^2}s^{-3/2}e^{-1/s}\dif s
\right]\nonumber\\
&\quad -\frac{\tau (1+(a+b)N^{1-\gamma})^3}{2(a+b)^2\sqrt{\theta}N^{3-2\gamma}}\int_0^{\frac{4\theta t N^{2}}{(1+(a+b)N^{1-\gamma})^2}}\int_0^{\frac{4\theta t N^{2}}{(1+(a+b)N^{1-\gamma})^2}}
(s_1+s_2)^{-5/2}e^{-\frac{1}{s_1+s_2}}\dif s_1\dif s_2\nonumber\\
& \quad +\frac{\tau}{2(a+b)^2\sqrt{\theta} N^{3-2\gamma}}\int_0^{4\theta t N^2}\int_0^{4\theta t N^2}
(s_1+s_2)^{-5/2}e^{-\frac{1}{s_1+s_2}}\dif s_1\dif s_2,
\end{align}
and
\begin{align}\label{eq:EDuh-3}
\widetilde{A}_3&=\frac{4\tau\theta N^{2\gamma}}{3(a+b)^2}(2t)^{3/2}\left(e^{-\frac{(1-(a+b)N^{1-\gamma})^2}{8\theta N^{2}t}}-e^{-\frac{1}{8\theta N^2t}}\right)
-\frac{8\tau\theta N^{2\gamma}}{3(a+b)^2}t^{3/2}\left(e^{-\frac{(1-(a+b)N^{1-\gamma})^2}{4\theta N^{2}t}}-e^{-\frac{1}{4\theta N^2t}}\right) \nonumber\\
&\quad +\frac{4\tau (1-(a+b)N^{1-\gamma})^2N^{2\gamma-2}}{3(a+b)^2}(2t)^{1/2}e^{-\frac{(1-(a+b)N^{1-\gamma})^2}{8\theta N^{2}t}}-\frac{4\tau N^{2\gamma-2}}{3(a+b)^2}(2t)^{1/2}e^{-\frac{1}{8\theta N^2t}} \nonumber\\
&\quad -\frac{8\tau (1-(a+b)N^{1-\gamma})^2N^{2\gamma-2}}{3(a+b)^2}t^{1/2}e^{-\frac{(1-(a+b)N^{1-\gamma})^2}{4\theta N^{2}t}}+\frac{8\tau N^{2\gamma-2}}{3(a+b)^2}t^{1/2}e^{-\frac{1}{4\theta N^2t}} \nonumber\\
& \quad - \frac{2\tau\left|1-(a+b)N^{1-\gamma}\right|^3}{3(a+b)^2\sqrt{\theta}N^{3-2\gamma}}\left[\int_{\frac{4\theta t N^{2}}{(1-(a+b)N^{1-\gamma})^2}}^{\frac{8\theta t N^{2}}{(1-(a+b)N^{1-\gamma})^2}}
s^{-3/2}e^{-1/s}\dif s-\int_0^{\frac{4\theta t N^{2}}{(1-(a+b)N^{1-\gamma})^2}}s^{-3/2}e^{-1/s}\dif s
\right]\nonumber\\
&\quad +\frac{2\tau}{3(a+b)^2\sqrt{\theta}N^{3-2\gamma}}\left[\int_{4\theta t N^2}^{8\theta t N^2}
s^{-3/2}e^{-1/s}\dif s-\int_0^{4\theta t N^2}s^{-3/2}e^{-1/s}\dif s
\right]\nonumber\\
&\quad -\frac{\tau \left|1-(a+b)N^{1-\gamma}\right|^3}{2(a+b)^2\sqrt{\theta}N^{3-2\gamma}}\int_0^{\frac{4\theta t N^{2}}{(1-(a+b)N^{1-\gamma})^2}}\int_0^{\frac{4\theta t N^{2}}{(1-(a+b)N^{1-\gamma})^2}}
(s_1+s_2)^{-5/2}e^{-\frac{1}{s_1+s_2}}\dif s_1\dif s_2\nonumber\\
& \quad +\frac{\tau}{2(a+b)^2\sqrt{\theta}N^{3-2\gamma}}\int_0^{4\theta t N^2}\int_0^{4\theta t N^2}
(s_1+s_2)^{-5/2}e^{-\frac{1}{s_1+s_2}}\dif s_1\dif s_2.
\end{align}
Using the asymptotics 
$$
\left(1-e^{-\frac{\alpha}{N^{2}}}\right) \simeq \alpha N^{-2}, \quad 
\left(e^{-\frac{\alpha}{N^{2}}}-e^{-\frac{\beta}{N^{2}}}\right)\simeq
(\beta-\alpha)N^{-2},
$$
with $\beta>\alpha$, and sufficiently large $N$,  we continue 
\begin{align*}
B_2 &\simeq
\frac{1}{N^{\gamma}}\left[-\frac{4\tau(a+b)}{3\sqrt{\theta}}\int_0^{\infty} s^{-3/2}e^{-1/s}\dif s
+\frac{\tau (a+b)}{\sqrt{\theta}}\int_0^{\infty}\int_0^{\infty}
(s_1+s_2)^{-5/2}e^{-\frac{1}{s_1+s_2}}\dif s_1\dif s_2\right]\\
&+\frac{1}{N^{3-2\gamma}}\Bigg[\frac{2\tau(1+(a+b)N^{1-\gamma})^3}{3(a+b)^2\sqrt{\theta}}\int_0^{\infty}s^{-3/2}e^{-1/s}\dif s
+\frac{2\tau\left|1-(a+b)N^{1-\gamma}\right|^3}{3(a+b)^2\sqrt{\theta}}\int_0^{\infty}s^{-3/2}e^{-1/s}\dif s\\
&-\frac{\tau (1+(a+b)N^{1-\gamma})^3}{2(a+b)^2\sqrt{\theta}}\int_0^{\infty}\int_0^{\infty}
(s_1+s_2)^{-5/2}e^{-\frac{1}{s_1+s_2}}\dif s_1\dif s_2\nonumber\\
&-\frac{\tau \left|1-(a+b)N^{1-\gamma}\right|^3}{2(a+b)^2\sqrt{\theta}}\int_0^{\infty}\int_0^{\infty}
(s_1+s_2)^{-5/2}e^{-\frac{1}{s_1+s_2}}\dif s_1\dif s_2\nonumber\\
&-\frac{4\tau}{3(a+b)^2\sqrt{\theta}}\int_0^{\infty}s^{-3/2}e^{-1/s}\dif s
+\frac{\tau}{(a+b)^2\sqrt{\theta}}\int_0^{\infty}\int_0^{\infty}
(s_1+s_2)^{-5/2}e^{-\frac{1}{s_1+s_2}}\dif s_1\dif s_2
\Bigg], 
\end{align*}
for large  $N$.
From here, using that 
$$
\int_0^{\infty} s^{-3/2}e^{-1/s}\dif s=\sqrt{\pi}  \quad
\mbox{and}\quad \int_0^{\infty}\int_0^{\infty} (s_1+s_2)^{-5/2}e^{-\frac{1}{s_1+s_2}}\dif s_1\dif s_2=\sqrt{\pi},
$$
the identity \eqref{eq:AsympDu} follows. 


\subsubsection{Proof of ~\eqref{eq:LimTildeQC2}-\eqref{eq:dTV2}}\label{proof:eqs:tildeQE-CLT1}
Without loss of generality, for computational simplicity, we assume that $a+b=1$.
Using Lemma~\ref{lemma:intbyparts-1}, split function $F$ defined in \eqref{eq:tildeQND} as
$$
F(i,t_k,t_{\ell})=:F_1(i,t_k,t_{\ell})+F_2(i,t_k,t_{\ell})+F_3(i,t_k,t_{\ell})+F_4(i,t_k,t_{\ell}),
$$
where
\begin{align*}
F_1(i,t_k,t_{\ell}) & 
=\frac{4}{3}(t_k+t_{\ell})^{3/2}\left[2e^{-\frac{i^2}{4\theta N^2(t_k+t_{\ell})}} -
e^{-\frac{(i+N^{1-\gamma})^2}{4\theta N^2(t_k+t_{\ell})}}-e^{-\frac{(i-N^{1-\gamma})^2}{4\theta N^2(t_k+t_{\ell})}}\right]\\
& -\frac{4}{3}t_k^{3/2}\left[2e^{-\frac{i^2}{4\theta N^2t_k}}-
e^{-\frac{(i+N^{1-\gamma})^2}{4\theta N^2t_k}}-e^{-\frac{(i-N^{1-\gamma})^2}{4\theta N^2t_k}}\right]\\
&-\frac{4}{3}t_{\ell}^{3/2}\left[2e^{-\frac{i^2}{4\theta N^2t_{\ell}}}-
e^{-\frac{(i+N^{1-\gamma})^2}{4\theta N^2t_{\ell}}}-e^{-\frac{(i-N^{1-\gamma})^2}{4\theta N^2t_{\ell}}}\right]
\end{align*}
\begin{align*}
F_2(i,t_k,t_{\ell})&=\frac{8i^2}{3\theta N^2}(t_k+t_{\ell})^{1/2}e^{-\frac{i^2}{4\theta N^2(t_k+t_{\ell})}}
-\frac{4(i+N^{1-\gamma})^2}{3\theta N^2}(t_k+t_{\ell})^{1/2}e^{-\frac{(i+N^{1-\gamma})^2}{4\theta N^2(t_k+t_{\ell})}}\\
&-\frac{4(i-N^{1-\gamma})^2}{3\theta N^2}(t_k+t_{\ell})^{1/2}e^{-\frac{(i-N^{1-\gamma})^2}{4\theta N^2(t_k+t_{\ell})}}
-\frac{8i^2}{3\theta N^2}t_k^{1/2}e^{-\frac{i^2}{4\theta N^2t_k}}\\
&+\frac{4(i+N^{1-\gamma})^2}{3\theta N^2}t_k^{1/2}e^{-\frac{(i+N^{1-\gamma})^2}{4\theta N^2t_k}}
+\frac{4(i-N^{1-\gamma})^2}{3\theta N^2}t_k^{1/2}e^{-\frac{(i-N^{1-\gamma})^2}{4\theta N^2t_k}}
-\frac{8i^2}{3\theta N^2}t_{\ell}^{1/2}e^{-\frac{i^2}{4\theta N^2t_{\ell}}}\\
&+\frac{4(i+N^{1-\gamma})^2}{3\theta N^2}t_{\ell}^{1/2}e^{-\frac{(i+N^{1-\gamma})^2}{4\theta N^2t_{\ell}}}
+\frac{4(i-N^{1-\gamma})^2}{3\theta N^2}t_{\ell}^{1/2}e^{-\frac{(i-N^{1-\gamma})^2}{4\theta N^2t_{\ell}}},
\end{align*}
\begin{align*}
F_3(i,t_k,t_{\ell})&=-\frac{4i^3}{3\theta^{3/2}N^3}
\left[\int_{\frac{4\theta N^2 t_{\ell}}{i^2}}^{\frac{4\theta N^2 (t_k+t_{\ell})}{i^2}}
s^{-3/2}e^{-1/s}\dif s-\int_0^{\frac{4\theta N^2 t_k}{i^2}}s^{-3/2}e^{-1/s}\dif s\right]\\
&+\frac{2(i+N^{1-\gamma})^3}{3\theta^{3/2}N^3}\left[\int_{\frac{4\theta N^2 t_{\ell}}{(i+N^{1-\gamma})^2}}^{\frac{4\theta N^2 (t_k+t_{\ell})}{(i+N^{1-\gamma})^2}}
s^{-3/2}e^{-1/s}\dif s-\int_0^{\frac{4\theta N^2 t_k}{(i+N^{1-\gamma})^2}}s^{-3/2}e^{-1/s}\dif s\right]\\
&+\frac{2(i-N^{1-\gamma})^3}{3\theta^{3/2}N^3}\left[\int_{\frac{4\theta N^2 t_{\ell}}{(i-N^{1-\gamma})^2}}^{\frac{4\theta N^2 (t_k+t_{\ell})}{(i-N^{1-\gamma})^2}}
s^{-3/2}e^{-1/s}\dif s-\int_0^{\frac{4\theta N^2 t_k}{(i-N^{1-\gamma})^2}}s^{-3/2}e^{-1/s}\dif s\right],
\end{align*}
\begin{align*}
F_4(i,t_k,t_{\ell})&=-\frac{i^3}{\theta^{3/2}N^3}
\int_0^{\frac{4\theta N^2 t_k}{i^2}}\int_0^{\frac{4\theta N^2 t_{\ell}}{i^2}}
(s_1+s_2)^{-5/2}e^{-\frac{1}{s_1+s_2}}\dif s_2\dif s_2\\
&+\frac{(i+N^{1-\gamma})^3}{2\theta^{3/2}N^3}\int_0^{\frac{4\theta N^2 t_k}{(i+N^{1-\gamma})^2}}
\int_0^{\frac{4\theta N^2 t_{\ell}}{(i+N^{1-\gamma})^2}}
(s_1+s_2)^{-5/2}e^{-\frac{1}{s_1+s_2}}\dif s_2\dif s_2\\
&+\frac{(i-N^{1-\gamma})^3}{2\theta^{3/2}N^3}\int_0^{\frac{4\theta N^2 t_k}{(i-N^{1-\gamma})^2}}
\int_0^{\frac{4\theta N^2 t_{\ell}}{(i-N^{1-\gamma})^2}}
(s_1+s_2)^{-5/2}e^{-\frac{1}{s_1+s_2}}\dif s_2\dif s_2.
\end{align*}
Similarly to \eqref{phi-taylor}, one can show that,  for each $i,k,\ell$,
\begin{equation*}
\left|2F_j(i,t_k,t_{\ell})-F_j(i+1,t_k,t_{\ell})-F_j(i-1,t_k,t_{\ell})\right|\leq \frac{C_1}{N^{2\gamma+2}},\quad j=1,2,3,4,
\end{equation*}
which imply that 
\begin{equation}\label{eq:tilPhibdd}
\left|2F(i,t_k,t_{\ell})-F(i+1,t_k,t_{\ell})-F(i-1,t_k,t_{\ell})\right|\leq \frac{C_1}{N^{2\gamma+2}}, 
\quad \mbox{and}\quad
\left|\widetilde{\Phi}_N(i,,t_k,t_{\ell})\right|\leq \frac{C_2}{N^2}.
\end{equation}
Using these, \eqref{eq:LimTildeQC2}, \eqref{eq:LimTildeQND1}, and \eqref{eq:LimTildeQND2} clearly follow. 

Next we will prove \eqref{eq:dTV2}, starting with computation of the  Malliavin derivative of 
$\mathcal{D}\widetilde{\mathcal{Q}}_{N,M}$
$$
\mathcal{D}\widetilde{\mathcal{Q}}_{N,M}=\frac{2}{M\sqrt{2N}}\sum_{j=1}^M\sum_{i=1}^{N} \widetilde{I}^{t_j}_1\left(1_{[x_{i-1},x_i]}\right)1_{[x_{i-1},x_i]}.
$$
Therefore,
\begin{align*}
\left\|\mathcal{D}\widetilde{\mathcal{Q}}_{N,M}\right\|_{\widetilde{\mathcal{H}}}^2&=
\frac{2}{NM^2}\sum_{k,\ell=1}^M\left(\mathbb{E}\widetilde{U}(t_k)\mathbb{E}\widetilde{U}(t_{\ell})\right)^{{-1/2}} 
\sum_{i,j=1}^{N} \widetilde{I}^{t_k}_1\left(1_{[x_{i-1},x_i]}\right) \widetilde{I}^{t_{\ell}}_1\left(1_{[x_{j-1},x_j]}\right)\\
&\quad \times \mathbb{E}\left(\Delta_{\frac{1}{N^{\gamma}}}u(t_k,x_i)-\Delta_{\frac{1}{N^{\gamma}}}u(t_k,x_{i-1})\right)
\left(\Delta_{\frac{1}{N^{\gamma}}}u(t_{\ell},x_j)-\Delta_{\frac{1}{N^{\gamma}}}u(t_{\ell},x_{j-1})\right)\\
&=\frac{2}{NM^2}\sum_{k,\ell=1}^M\left(\mathbb{E}\widetilde{U}(t_k)\mathbb{E}\widetilde{U}(t_{\ell})\right)^{{-1/2}} 
\sum_{i,j=1}^{N}  \widetilde{I}^{t_k,t_{\ell}}_2\left(1_{[x_{i-1},x_i]}\otimes 1_{[x_{j-1},x_j]}\right) \\
&\quad\times \mathbb{E}\left(\Delta_{\frac{1}{N^{\gamma}}}u(t_k,x_i)-\Delta_{\frac{1}{N^{\gamma}}}u(t_k,x_{i-1})\right)
\left(\Delta_{\frac{1}{N^{\gamma}}}u(t_{\ell},x_j)-\Delta_{\frac{1}{N^{\gamma}}}u(t_{\ell},x_{j-1})\right) \\ 
& \qquad +\mathbb{E}\left\|\mathcal{D}\widetilde{\mathcal{Q}}_{N,M}\right\|_{\widetilde{\mathcal{H}}}^2.
\end{align*}
Then,
\begin{align*}
\widetilde{\mathcal{R}}_{N,M}&:=\left\|\mathcal{D}\widetilde{\mathcal{Q}}_{N,M}\right\|_{\widetilde{\mathcal{H}}}^2-\mathbb{E}\left\|\mathcal{D}\widetilde{\mathcal{Q}}_{N,M}\right\|_{\widetilde{\mathcal{H}}}^2\\
&=\frac{2}{NM^2}\sum_{k,\ell=1}^M\left(\mathbb{E}\widetilde{U}(t_k)\mathbb{E}\widetilde{U}(t_{\ell})\right)^{{-1/2}} 
\sum_{i,j=1}^{N}  \widetilde{I}^{t_k,t_{\ell}}_2\left(1_{[x_{i-1},x_i]}\otimes 1_{[x_{j-1},x_j]}\right) \\
& \quad \times \mathbb{E}\left(\Delta_{\frac{1}{N^{\gamma}}}u(t_k,x_i)-\Delta_{\frac{1}{N^{\gamma}}}u(t_k,x_{i-1})\right)
\left(\Delta_{\frac{1}{N^{\gamma}}}u(t_{\ell},x_j)-\Delta_{\frac{1}{N^{\gamma}}}u(t_{\ell},x_{j-1})\right),
\end{align*}
and
\begin{align*}
\mathbb{E}\widetilde{\mathcal{R}}^2_{N,M}&
=\frac{8}{N^2M^4}\sum_{i',j',k',\ell'=1}^M \sum_{i,j,k,l=1}^{N} \left(
\mathbb{E}\widetilde{U}(t_{i'})\mathbb{E}\widetilde{U}(t_{j'})\mathbb{E}\widetilde{U}(t_{k'})
\mathbb{E}\widetilde{U}(t_{\ell'})\right)^{-1}\nonumber\\
&\times \mathbb{E}\left(\Delta_{\frac{1}{N^{\gamma}}}u(t_{i'},x_i)-\Delta_{\frac{1}{N^{\gamma}}}u(t_{i'},x_{i-1})\right)
\left(\Delta_{\frac{1}{N^{\gamma}}}u(t_{k'},x_k)-\Delta_{\frac{1}{N^{\gamma}}}u(t_{k'},x_{k-1})\right)\nonumber\\
&\times \mathbb{E}\left(\Delta_{\frac{1}{N^{\gamma}}}u(t_{j'},x_j)-\Delta_{\frac{1}{N^{\gamma}}}u(t_{j'},x_{j-1})\right)
\left(\Delta_{\frac{1}{N^{\gamma}}}u(t_{\ell'},x_{\ell})-\Delta_{\frac{1}{N^{\gamma}}}u(t_{\ell'},x_{\ell-1})\right)
\nonumber\\
&\times \mathbb{E}\left(\Delta_{\frac{1}{N^{\gamma}}}u(t_{i'},x_i)-\Delta_{\frac{1}{N^{\gamma}}}u(t_{i'},x_{i-1})\right)
\left(\Delta_{\frac{1}{N^{\gamma}}}u(t_{j'},x_j)-\Delta_{\frac{1}{N^{\gamma}}}u(t_{j'},x_{j-1})\right)
\nonumber\\
&\times \mathbb{E}\left(\Delta_{\frac{1}{N^{\gamma}}}u(t_{k'},x_k)-\Delta_{\frac{1}{N^{\gamma}}}u(t_{k'},x_{k-1})\right)
\left(\Delta_{\frac{1}{N^{\gamma}}}u(t_{\ell'},x_{\ell})-\Delta_{\frac{1}{N^{\gamma}}}u(t_{\ell'},x_{\ell-1})\right).
\end{align*}
Consequently, 
\begin{align*}
\mathbb{E}\widetilde{\mathcal{R}}^2_{N,M} & 
=\frac{8}{N^2M^4}\sum_{i',j',k',\ell'=1}^M \sum_{i,j,k,l=1}^{N} \left(
\mathbb{E}\widetilde{U}(t_{i'})\mathbb{E}\widetilde{U}(t_{j'})\mathbb{E}\widetilde{U}(t_{k'})
\mathbb{E}\widetilde{U}(t_{\ell'})\right)^{-1}\\
& \quad \times\widetilde{\Phi}_N(|i-k|,t_{i'},t_{k'})\widetilde{\Phi}_N(|j-\ell|,t_{j'},t_{\ell'})
\widetilde{\Phi}_N(|i-j|,t_{i'},t_{j'})\widetilde{\Phi}_N(|k-l|,t_{k'},t_{j'}).
\end{align*}
Note that by \eqref{eq:tilPhibdd}, for each $i,k,\ell$,
\begin{align*}
\widetilde{\Phi}_N(i,t_k,t_{\ell})\leq \frac{C_1}{N^2}, \quad
\mathbb{E}\widetilde{\mathcal{R}}^2_{N,M}
\leq \frac{C_2}{N^2}, 
\end{align*}
for some $C_1,C_2>0$. Therefore, in view of Theorem~\ref{MSthm},
\begin{equation*}
d_{TV}\left(\widetilde{\mathcal{Q}}_{N,M}, \mathcal{N}(0,1)\right)\leq C_3\left(\mathbb{E}\widetilde{\mathcal{R}}_{N,M}^2\right)^{1/2}\leq \frac{C_4}{N}, 
\end{equation*}
for some $C_3,C_4>0$. This concludes the proof. 

\subsection{Auxiliary results}

\begin{lemma}\label{lemma:intbyparts-1}
Let $c,t,t'\geq 0$. Then,
\begin{align*}
&\int_0^{t'}\int_0^t (s_1+s_2)^{-1/2}e^{-\frac{c}{s_1+s_2}}\dif s_1\dif s_2=\frac{4}{3}
\left((t'+t)^{3/2}e^{-\frac{c}{t'+t}}-(t')^{3/2}e^{-\frac{c}{t'}}-t^{3/2}e^{-\frac{c}{t}}\right)\\
&\  +\frac{16c}{3}\left((t'+t)^{1/2}e^{-\frac{c}{t'+t}}-(t')^{1/2}e^{-\frac{c}{t'}}-t^{1/2}e^{-\frac{c}{t}}\right)
-\frac{16c^{3/2}}{3}\left(\int_{\frac{t}{c}}^{\frac{t'+t}{c}} s^{-3/2}e^{-\frac{1}{s}}\dif s-\int_0^{\frac{t'}{c}} s^{-3/2}e^{-\frac{1}{s}}\dif s
\right)\\
& \quad  -4c^{3/2}\int_0^{\frac{t'}{c}}\int_0^{\frac{t}{c}} (s_1+s_2)^{-5/2} e^{-\frac{1}{s_1+s_2}}\dif s_1\dif s_2,
\end{align*}
and
\begin{align*}
&\int_0^t\int_0^t (s_1+s_2)^{-3/2}e^{-\frac{c}{s_1+s_2}}\dif s_1\dif s_2=
4\left((t')^{1/2}e^{-\frac{c}{t'}}+t^{1/2}e^{-\frac{c}{t}}-(t'+t)^{1/2}e^{-\frac{c}{t'+t}}\right)\\
& \qquad +4c^{1/2}\left(\int_{\frac{t}{c}}^{\frac{t'+t}{c}} s^{-3/2}e^{-\frac{1}{s}}
-\int_0^{\frac{t'}{c}} s^{-3/2}e^{-\frac{1}{s}}
\right)+2c^{1/2}\int_0^{\frac{t'}{c}}\int_0^{\frac{t}{c}} (s_1+s_2)^{-5/2}e^{-\frac{1}{s_1+s_2}}\dif s_1\dif s_2.
\end{align*}
\end{lemma}
\begin{proof}
Both identities follow directly after integration by parts. 

\end{proof}

As already mentioned, the proof of asymptotic normality of $\cQ_{N,M}$ and $\wt\cQ_{N,M}$
relies on Stein's method and Malliavin calculus. For brevity, we list here only two relevant results from this topic, and refer the reader to \cite{Nualart2006,NourdinPeccati2012} for details. Alternatively, the reader may see \cite[Appendix~B]{CialencoDelgado-VencesKim2019} for a comprehensive snapshot of elements of Malliavin calculus used in the current paper. 
 
\begin{theorem}\label{MSthm} \cite[Theorem 5.2.6]{NourdinPeccati2012}
	Let $\mathcal{U}$ be the canonical Hilbert space associated to a Gaussian process.
	Let $I_q(f)$ be a multiple integral of order $q\geq 2$ with respect to the Gaussian process. Assume that 
	$\mathbb{E}\left[I_q(f)^2\right]=\sigma^2$. Then,
	$$d_{TV}\left(I_q(f),\mathcal{N}\left(0,\sigma^2\right)
	\right)\leq \frac{2}{\sigma^2}\left(\mbox{Var}\left[\frac{1}{q}\|\mathcal{D}I_q(f)\|^2_{\mathcal{U}}
	\right]
	\right)^{1/2},$$
	where $d_{TV}$ is the total variation distance and $\mathcal{N}\left(0,\sigma^2\right)$ is a Gaussian random variable with mean 0 and variance $\sigma^2$.
\end{theorem}

\begin{theorem}\label{MSthm2} \cite[Corollary 5.2.8]{NourdinPeccati2012}
	Let $F_N=I_q\left(f_N\right)$, $N\geq 1$, be a sequence of multiple integrals of order $q\geq 2$ with respect to a Gaussian process.
	Assume that $\mathbb{E}\left[F_N^2\right]\rightarrow \sigma^2>0$, as $N\rightarrow\infty$.
	Then, the following assertions are equivalent:
	\begin{enumerate}
		\item $\displaystyle\w\lim_{N\to \infty}F_N=\mathcal{N}\left(0,\sigma^2\right)$;
		\item $\displaystyle\lim_{N\to \infty} d_{TV}\left(F_N,\mathcal{N}\left(0,\sigma^2\right)\right)=0$.
	\end{enumerate}
\end{theorem}

%

\def\cprime{$'$}

\end{document}